%% file: AH_mass_draft_revision_arxiv.tex
\documentclass[12pt]{amsart}
\usepackage{amssymb,amsmath,amsthm,latexsym}
\usepackage{amsrefs,mathrsfs,cancel,hyperref,verbatim}
\usepackage{pdfpages}
\usepackage{subcaption}
\usepackage[foot]{amsaddr}

\author{Hyun Chul Jang}
\author{Pengzi Miao}
\address{Department of Mathematics, University of Miami}
\email{h.jang@math.miami.edu, pengzim@math.miami.edu}
\title{Hyperbolic Mass via Horospheres}

\begin{document}

%%%%%%%%%%%My Definitions %%%%%%%%%%%%%%%%%%%%%%%%%%%%%%%%%%%%%%%%%
\newtheorem{theorem}{Theorem} %%%% or [chapter]
\newtheorem{claim}[theorem]{Claim}
\newtheorem{lemma}[theorem]{Lemma}
\newtheorem{proposition}[theorem]{Proposition}
\newtheorem{corollary}[theorem]{Corollary}
\theoremstyle{definition}
\newtheorem{definition}[theorem]{Definition}
\theoremstyle{remark}
\newtheorem{remark}[theorem]{Remark}

\numberwithin{equation}{section}
\numberwithin{theorem}{section}
\numberwithin{figure}{section}

\begin{abstract}
	We derive geometric formulas for the mass of asymptotically hyperbolic manifolds using coordinate horospheres.  
	%and state relevant rigidity results. 
	%Applications of 
	As an application, we obtain a new rigidity result of hyperbolic space: if a complete asymptotically hyperbolic manifold has scalar curvature lower bound $-n(n-1)$ and is isometric to hyperbolic space outside a coordinate horosphere, then the manifold is isometric to hyperbolic space. In addition, we apply our formula to investigate regions near infinity that do not contribute to the mass quantity, which leads to improved rigidity results of hyperbolic space.
\end{abstract}

\maketitle
%\tableofcontents

\section{Introduction}\label{sec:introduction}

The mass of an asymptotically hyperbolic Riemannian manifold introduced in \cites{Wang.X:2001,Chrusciel-Herzlich:2003} serves as a global geometric invariant that measures the deviation from hyperbolic space. In this paper, we derive geometric formulas of the mass for asymptotically hyperbolic manifolds using large coordinate horospheres, and state some relevant rigidity results.

Let $(\mathbb{H}^n,b)$ be hyperbolic space as the upper sheet of the hyperboloid in the Minkowski space:
\[
	\mathbb{H}^n=\{(z,t)\in\mathbb{R}^{n,1}:z_1^2+\cdots+z_n^2-t^2=-1, t>0\}.
\]
Let $r=\sqrt{z_1^2+\cdots+z_n^2}$, then the metric $b$ is written as
\[
	b=\frac{dr^2}{1+r^2}+r^2g_{\mathbb{S}^{n-1}}.
\]

Following \cite{Chrusciel-Herzlich:2003}, a Riemannian manifold $(M^n,g)$ is said to be asymptotically hyperbolic if there exist a compact set $K\subset M$ and a diffeomorphism $\Phi:M\setminus K\to \mathbb{H}^n\setminus B_R(0)$, where $B_R(0)=\{r<R\}$, such that $h:=(\Phi^{-1})^* g-b$ satisfies
\begin{enumerate}
	\item as $r\to\infty$,
	\[
		|h|_b+|\mathring{\nabla}h|_b+|\mathring{\nabla}^2 h|_b=O(r^{-q}),\quad q>\frac{n}{2},
	\]
	where $\mathring{\nabla}$ is the covariant derivative with respect to $b$.
	\item $\int_{M\setminus K} r(R_g+n(n-1)) d\mu_g <\infty$ where $R_g$ and $d\mu_g$ are the scalar curvature and the volume element of $g$, respectively.
\end{enumerate}
The mass of asymptotically hyperbolic manifolds is defined as a $(n+1)$-vector using the so-called mass integral, which is defined as the following limit of the integral on large coordinate spheres:
\begin{equation}\label{eqn:AH mass integral}
	H_{\Phi}(V)=\lim_{r\to\infty}\int_{S_r} (V\,\mathrm{div }_b h-V\,d(\mathrm{tr}_{b} h)+(\mathrm{tr}_{b} h)d V-h(\mathring{\nabla} V,\cdot))(\nu_0)\,d\sigma_b
\end{equation}
for $V\in \mathcal{S}(\mathbb{H}^n)$, where $\mathcal{S}(\mathbb{H}^n)$ is the space of static potentials on $\mathbb{H}^n$, given by
\[
	\mathcal{S}(\mathbb{H}^n)=\textrm{span}\{t,z_1,\ldots,z_n\},\textrm{ where }t=\sqrt{1+r^2}.
\]
The components of the mass vector $(p_0,p_1,\ldots,p_n)$ of $(M^n,g)$ are defined as
\[
	p_0=H_{\Phi}(t),\quad p_i=H_{\Phi}(z_i) \text{ for }i=1,\ldots,n.
\]
Whereas its components may depend on the exterior coordinate chart $\Phi$, the (squared)-Minkowskian length of the mass vector 
\[
	m(g)^2=p_0^2-\sum_{i=1}^n p_i^2
\]
is a geometric invariant. The Riemannian positive mass theorem for asymptotically hyperbolic manifolds states that if $R_g\ge -n(n-1)$, then $p_0^2\ge\sum_{i=1}^n p_i^2$, and equality holds if and only if $(M,g)$ is isometric to $\mathbb{H}^n$. This theorem was first proved in \cites{Wang.X:2001,Chrusciel-Herzlich:2003} under spinor assumption. In \cite{Andersson:2008du}, the spinor assumption was replaced by the restriction on dimension and the geometry at infinity. These assumptions have recently been removed in \cites{Chrusciel:2018co,Chrusciel.2019,Huang.2020}. We also remark that Sakovich \cite{Sakovich:2020aa} gave another proof for the $3$-dimensional case using the Jang equation.

To state our main theorems, let $\mathcal{H}_L$ for large enough $L>0$ denote a horosphere 
\[
	\mathcal{H}_L=\{z\in\mathbb{H}^n:\sqrt{1+r^2}-z_1=e^L\}.
\]
Note that $\mathcal{H}_L=\{y\in\mathbb{H}^n: y_1=e^{-L}\}$ where $\{y_1,\ldots,y_n\}$ is the half-space-model coordinates (see Figure \ref{fig:horospheres}). Geometrically, horospheres are complete noncompact (strongly) stable CMC hypersurfaces in $\mathbb{H}^n$ with mean curvature $n-1$.

\begin{figure}[ht]
	\centering
	\fontsize{9.5pt}{9.5pt}\selectfont
	\def\svgwidth{0.98\textwidth}%
	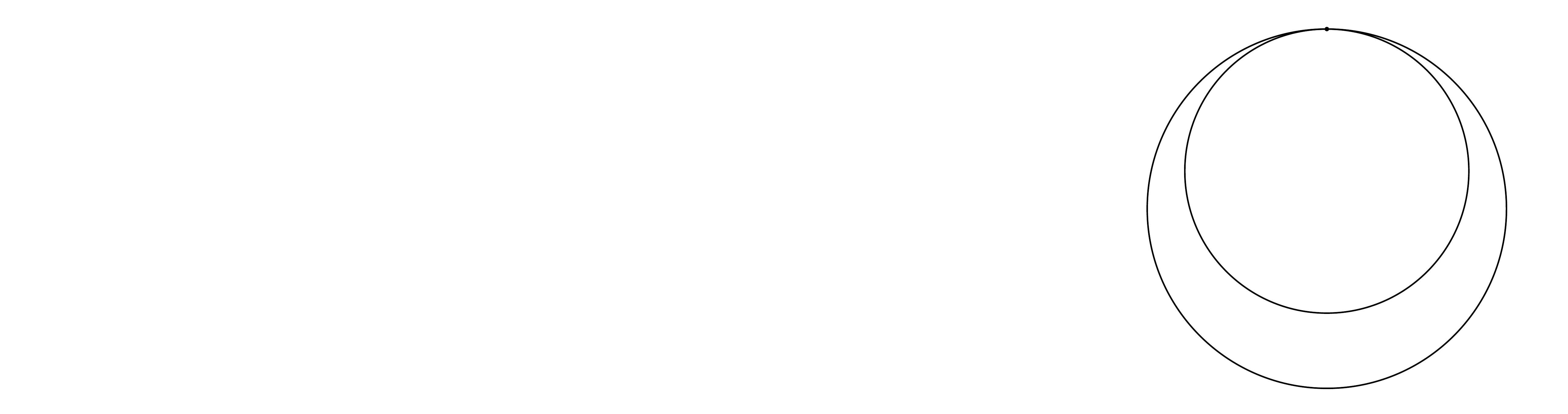%
	\caption{$\mathcal{H}_L$ in the hyperboloid, half-space, and ball models}\label{fig:horospheres}
\end{figure}

\begin{theorem}\label{thm:AH_mass_horospheres}
	Let $(M^n,g),n\ge 3,$ be an asymptotically hyperbolic manifold with metric falloff rate $q>\frac{n}{2}$. Let $\nu_g$ denote the unit normal vector to $\mathcal{H}_L$ in $(M,g)$ pointing toward $\{y_1=0\}$ (see Figure \ref{fig:horospheres}). Let $H_g$ be the mean curvature of $\mathcal{H}_L$ with respect to $\nu_g$ in $(M,g)$. Then, as $L\to\infty$,
	\begin{equation}\label{eqn:AH_mass_horospheres}
		p_0-p_1=2\int_{\mathcal{H}_L} V(H_b-H_g)\, d\sigma_g+o(1)
	\end{equation}
	where $p_0,\, p_1$ are the components of the mass vector, $V=\sqrt{1+r^2}-z_1=\frac{1}{y_1}$ on $M\setminus K$ and $H_b=n-1$ is the mean curvature of horospheres in hyperbolic space.
\end{theorem}

The formula \eqref{eqn:AH_mass_horospheres} can also be used to compute more general expression of the mass vector as the following: let
\[
	V=\sqrt{1+r^2}-\sum_{i=1}^n a^iz_i,\quad \mathcal{H}_L=\{z\in\mathbb{H}^n:V=e^L\}
\]
where $(a^1)^2+\cdots +(a^n)^2=1$. Note that $\mathcal{H}_L$ represents the horospheres that are based at $\sum_{i=1}^n a^iz_i=+\infty$ in the conformal infinity.
Then, as $L\to\infty,$
\begin{equation}
	p_0-\sum_{i=1}a^ip_i=2\int_{\mathcal{H}_L}V(H_b-H_g)\, d\sigma_g+o(1).
\end{equation}
Especially, one can compute $p_0+ p_i$ by using $V=\sqrt{1+r^2}+ z_i$ and the corresponding horospheres which are based at the antipodal point of the base from the case of $p_0-p_i$. Therefore, each component of the mass vector is computable by Theorem \ref{thm:AH_mass_horospheres}.

Combining this formula and the positive mass theorem for asymptotically hyperbolic manifolds, the following rigidity can be obtained:
\begin{corollary}\label{cor:rigidity_horosphere}
	Let $(M^n,g),n\ge 3,$ be an asymptotically hyperbolic manifold with metric falloff rate $q>\frac{n}{2}$ and scalar curvature $R_g\ge -n(n-1)$. If $(M,g)$ is isometric to hyperbolic space outside a coordinate horoball, then $(M,g)$ is isometric to hyperbolic space.
\end{corollary}
Here, outside a coordinate horoball means a type of region $\{0<y_1<a\}$ for some constant $a>0$ in the half-space model. Besides being a consequence of the positive mass theorem, the scalar curvature rigidity of hyperbolic space with a compact set has been proved separately in the literature, see \cites{Min-Oo.1989,Andersson:1998bq,Andersson:2008du}. %Due to the noncompactness of a horoball ($\{y_1>a\}$), the above corollary cannot be seen directly from the rigidity of the positive mass theorem. 
If one consider horospheres in hyperbolic space as a natural analog of hyperplanes in Euclidean space, it is known that Euclidean space does not have this kind of rigidity. Indeed, it is proved by Carlotto and Schoen \cite{Carlotto:2016dh} that there exists a nontrivial asymptotically flat metric on $\mathbb{R}^n$ with nonnegative scalar curvature such that the metric is isometric to Euclidean space in $(0,\infty)\times\mathbb{R}^{n-1}\subset\mathbb{R}^n$.

%The above corollary also gives an insight in the context of the scalar curvature rigidity results. This scalar curvature rigidity on the standard hemisphere $S^n_+$ was conjectured by Min-Oo \cite{Min-Oo:1998aa}, but Brendle et al \cite{Brendle:2011aa} showed the conjecture is false. They proved the following: for $n\ge 3$, there exists a smooth metric $g$ on $S^n_+$ with $R_g\ge n(n-1)$ such that $g$ agrees with the standard metric $g_{\mathbb{S}^n_+}$ in a neighborhood of $\partial S^n_+$ and $R_g>n(n-1)$ at some point on $S^n_+$. For Euclidean space, Carlotto and Schoen \cite{Carlotto:2016dh} proved that if an asymptotically flat metric $g$ of nonnegative scalar curvature is not trivial, then the region where it is not flat must contain a cone of positive aperture. In contrast, Corollary \ref{cor:rigidity_horosphere} implies that hyperbolic space has such rigidity with scalar curvature $R_g\ge -n(n-1)$ not only for any compact set, but also for some noncompact regions such as a horoball.

%As a special case of the positive mass theorem for asymptotically flat manifolds, it is well-known that if a metric $g$ on $\mathbb{R}^n$ with scalar curvature $R_g\ge 0$ agrees with the Euclidean metric outside a compact set, then $(\mathbb{R}^n,g)$ is isometric to Euclidean space. 

The key idea of the proof of Theorem \ref{thm:AH_mass_horospheres} is to use a family of parabolic cylinders $C_L$ defined as
\[
 	C_L=\{(y_1,\hat{y})\in\mathbb{H}^n:e^{-L}\le y_1\le e^{L},|\hat{y}|\le\sigma(L)\}
\]
where $\sigma(L)$ is a positive, monotone increasing function of $L$ such that $\sigma(L)\to\infty$ as $L\to\infty$ (see Figure \ref{fig:cylinder}). In fact, we obtain a more refined mass formula which only uses one surrounding surface of $C_L$:

\begin{figure}[ht]
	\centering
%	\fontsize{9.5pt}{9.5pt}\selectfont
	\def\svgwidth{3.5in}%
	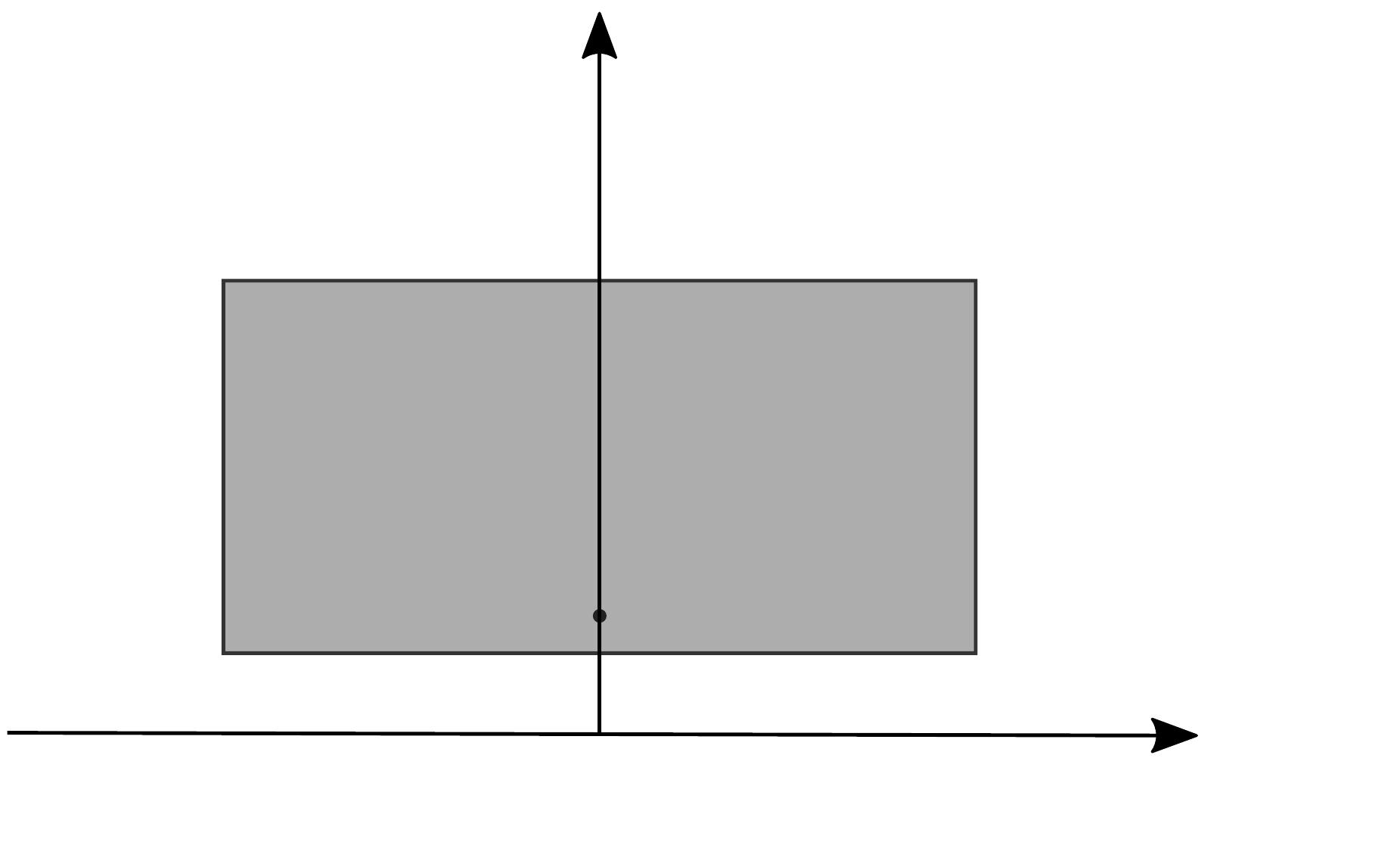%
	\caption{$\Sigma_L, C_L$ in the half-space model}\label{fig:cylinder}
\end{figure}

\begin{theorem}\label{thm:AH_mass_top_face}
	Let $(M^n,g),n\ge 3,$ be an asymptotically hyperbolic manifold with metric falloff rate $q>\frac{n}{2}$. Suppose that $\sigma(L)$ is a positive function of $L$ that tends to infinity as $L\to\infty$. If $q\le n-1$, we additionally assume that as $L\to\infty$
	\begin{equation}\label{eqn:sigma_condition}
		\sigma(L)^{n-2-2q}=\left\{
		\begin{array}{ll}
			o(e^{(q-n+1)L})& \text{ if } q<n-1,\vspace*{3pt}\\
			o(L^{-1})& \text{ if }q=n-1
		\end{array}\right.
	\end{equation}
	Define $\Sigma_L=\{y_1=e^{-L},|\hat{y}|<\sigma(L)\}$ (see Figure \ref{fig:cylinder}). Let $\nu_g$ denote the unit normal vector to $\Sigma_L$ in $(M,g)$ pointing toward $\{y_1=0\}$. Let $H_g$ be the mean curvature of $\Sigma_L$ with respect to $\nu_g$ in $(M,g)$. Then, as $L\to\infty$,
	\begin{equation}\label{eqn:AH_mass_top_face}
		p_0-p_1=2\int_{\Sigma_L} V(H_b-H_g)\, d\sigma_g+o(1)
	\end{equation}
	where $V=t-z_1=\frac{1}{y_1}$ on $M\setminus K$.
\end{theorem}

A few remarks relevant to this result are in order:

\begin{remark}
	The idea of using parabolic cylinders is motivated by the work in \cite{Miao.2019}, which provides the geometric mass formula for asymptotically flat metrics using large coordinate cubes. In a recent development, Bray et al. \cite{Bray.2019} presented a new proof of the Riemannian positive mass theorem for asymptotically flat $3$-metrics that gives an explicit lower bound for the mass in terms of linear growth harmonic functions and scalar curvature. Their approach was to apply Stern's integral formula \cite{Stern.2019} on level sets of a harmonic function. In this context, the mass formula using coordinate cubes in \cite{Miao.2019} can be viewed as to compute the mass by using the level sets of coordinate functions, which are linear growth harmonic functions on Euclidean space. This harmonic level set technique has been generalized for $3$-dimensional asymptotically flat initial data sets by Hirsch et al \cite{Hirsch.2020}. In particular, they defined spacetime harmonic functions on a given intial data set $(M^3,g,k)$ as a solution of the following equation
	\[
		\Delta u+(\mathrm{Tr}_g k)|\nabla u|=0,
	\] and established a generalized integral formula for such functions. From this point of view, the formula \eqref{eqn:AH_mass_horospheres} uses the level sets of a spacetime harmonic function $y_1^{-1}$ on hyperbolic space as an initial data set $(\mathbb{H}^3,b,-b)$.
\end{remark}

\begin{remark}
	It was also pointed out in \cite{Miao.2019} and \cite{Li.2019} that the cubic mass formula for asymptotically flat manifolds has connection with Gromov's scalar curvature comparison theory for cubic Riemannian polyhedra. In particular, Li \cite[Section 5]{Li.2019} observed that the so-called dihedral rigidity phenomenon for a Euclidean polyhedron $P$ is a localization of the positive mass theorem for asymptotically flat manifolds. Moreover, Li \cite{Li.2020} extended the dihedral rigidity for a collection of parabolic polyhedrons enclosed by horospheres in hyperbolic spaces. The use of parabolic cylinders in our proof is relatable to this, see Remark \ref{rmk:why_not_cubic} for a relevant discussion. %Therefore, one may naturally consider computing the mass of asymptotically hyperbolic manifolds by using large parabolic polyhedrons.
\end{remark}

\begin{remark}\label{rmk:mass region}
	Theorem \ref{thm:AH_mass_top_face} implies that the quantity $p_0-p_1$ is determined from the region 
	\[
		\bigcup_{L\ge L_0} \Sigma_{L}=\{(y_1,\hat{y})\in\mathbb{H}^n:0<y_1<e^{-L_0}, |\hat{y}|<f(y_1)\},
	\]
	for some large $L_0>0$. Here, $f(y_1)$ is any positive function of $y_1$ such that $f(y_1)\to\infty$ as $y_1\to 0$ and $f(y_1)^{n-2-2q}=o(y_1^{n-1-q})$ if $q<n-1$ or $f(y_1)^{-n}=o(e^{-1/L})$ if $q=n-1$. The latter condition is needed only if $\frac{n}{2}<q\le n-1$, so the shaded region in Figure \ref{fig:mass region} represents an example of $\bigcup_{L\ge L_0} \Sigma_{L}$ in such case. On the other hand, if $q> n-1$, such region can be arbitrarily thin as $|\hat{y}|\to \infty$ since we do not need any additional assumption on $\sigma(L)$ for Theorem \ref{thm:AH_mass_top_face}.
	\begin{figure}[ht]
		\centering
		\fontsize{11pt}{11pt}\selectfont
		\def\svgwidth{3in}%
		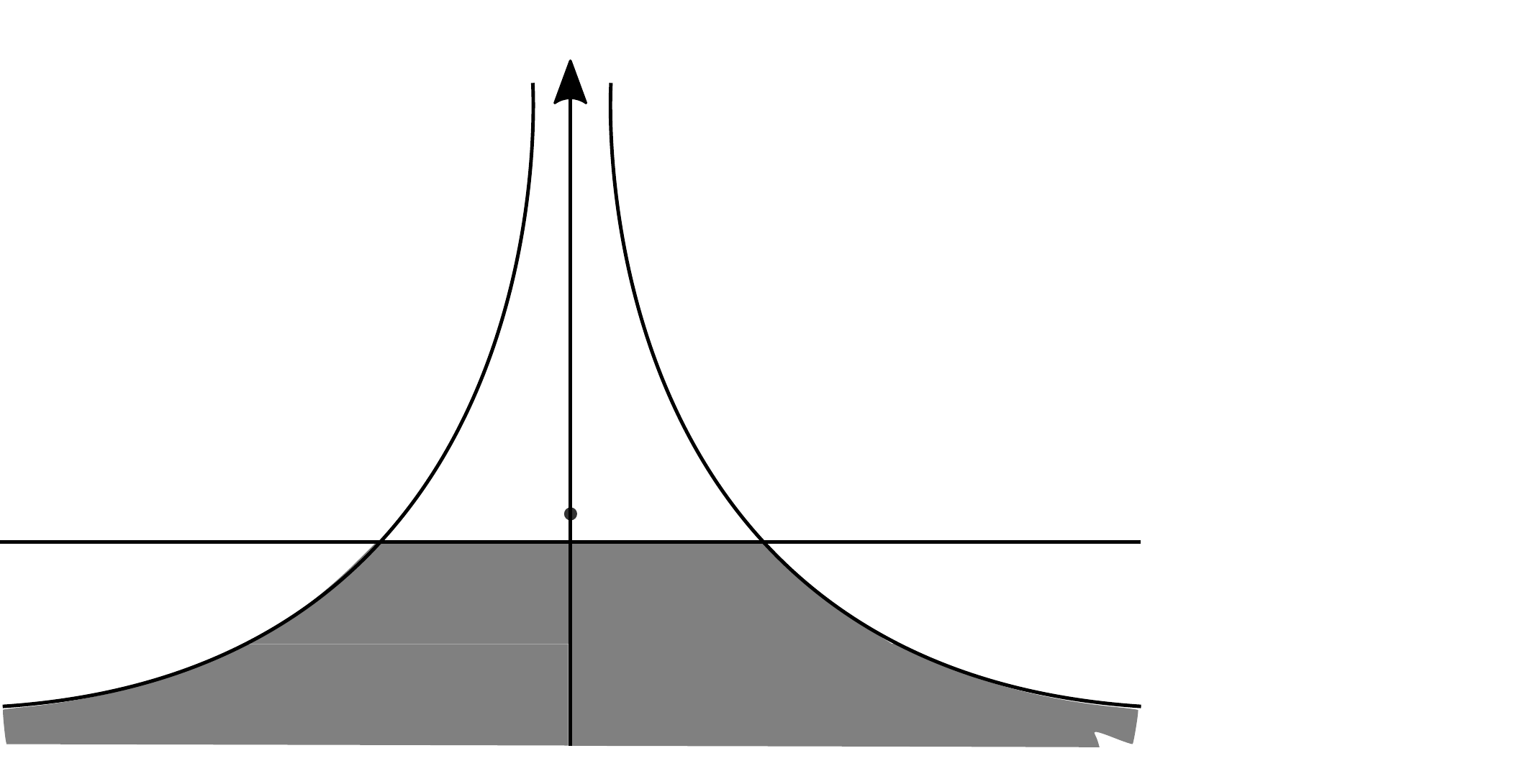		\caption{The region in Remark \ref{rmk:mass region}}\label{fig:mass region}
	\end{figure}
\end{remark}

Using Theorem \ref{thm:AH_mass_top_face} and the mass rigidity of asymptotically hyperbolic manifolds in \cite{Huang.2020}, we obtain the following rigidity result, which is stronger than Corollary \ref{cor:rigidity_horosphere}.

\begin{corollary}\label{cor:rigidity_general}
	Let $(M^n,g),n\ge 3,$ be an asymptotically hyperbolic manifold with the scalar curvature lower bound $R_g\ge -n(n-1)$. Define $\Sigma_L=\{y_1=e^{-L}, |\hat{y}|<\sigma(L)\}$ where $\sigma(L)$ satisfies \eqref{eqn:sigma_condition}. Suppose that there exists $L_0>0$ such that $H_g(\Sigma_L)\ge n-1$ for all $L\ge L_0$. Then $(M^n,g)$ is isometric to hyperbolic space $(\mathbb{H}^n,b)$. In particular, if $(M,g)$ is isometric to hyperbolic space in the region $\bigcup_{L>L_0}\Sigma_L$, then $(M,g)$ is isometric to hyperbolic space.
\end{corollary}

\begin{proof}
	By the assumption and Theorem \ref{thm:AH_mass_top_face}, we have
	\[
		p_0-p_1=\lim_{L\to\infty}2\int_{\Sigma_L} e^{L}(n-1-H_g)\, d\sigma_g\le 0.
	\]
	Combining the positivity of the mass (see \cites{Wang.X:2001,Chrusciel-Herzlich:2003,Chrusciel.2019}), we have 
	\[
		p_1\ge p_0\ge \left(\sum_{i=1}^n p_i^2\right)^{1/2}\ge |p_1|,
	\]which forces
	\[
		p_0=|p_1|=\left(\sum_{i=1}^n p_i^2\right)^{1/2}.
	\]
	Hence by the mass rigidity result from \cite{Huang.2020}, we obtain the first statement. The second statement follows from the fact that $H_g(\Sigma_L)=n-1$ for all $L\ge L_0$ if $(M,g)$ is isometric to hyperbolic space in the region $\bigcup_{L\ge L_0}\Sigma_L$.
\end{proof}
%In fact, Corollary \ref{cor:rigidity_general} implies Corollary \ref{cor:rigidity_horosphere} since the region $\bigcup_{L\ge L_0}\Sigma_L$ lies outside a horosphere $\mathcal{H}_{L_0}$.

We also give a sufficient condition for a region that does not contribute to the mass quantity.

\begin{proposition}\label{prop:nonlocalizable}
	 Let $(M^n,g),n\ge 3,$ be an asymptotically hyperbolic manifold with metric falloff rate $q>\frac{n}{2}$. Let $\Sigma_L=\{y_1=e^{-L},|\hat{y}|\le\sigma(L)\}$ where $\sigma(L)$ satisfies \eqref{eqn:sigma_condition}. For a given subset $U$ of $M$ and large $L>0$, define $\Theta(U,L)=e^{-L(n-1)}|U\cap\Sigma_L|_b$. Suppose 
	 \begin{equation}\label{eqn:theta}
	 	\Theta(U,L)=o(e^{L(q-n)}) \text{ as } L\to\infty
	 \end{equation}
	 Then, as $L\to\infty$,
	 \begin{equation}\label{eqn:size_infinity}
	 	p_0-p_1=2\int_{\Sigma_L\setminus U}V(H_b-H_g)\, d\sigma_g+o(1).
	 \end{equation}
	 In particular, if $q=n$, the above holds if 
	 \[
	 	\Theta(U):=\limsup_{L\to\infty}\Theta(U,L)=0.
	 \]
\end{proposition}

\begin{proof}
	By direct computation, we have
	\begin{equation}
		\begin{aligned}
			\int_{U\cap\Sigma_L} V(H_b-H_g)\, d\sigma_g &\le C_1\int_{U\cap\Sigma_L} e^L\left(\frac{2}{e^L+e^{-L}+e^L|\hat{x}|^2}\right)^{q}\, d\sigma_b\\
			&\le C_2 e^{L(1-q)}|U\cap\Sigma_L|_b\\
			&\le \tilde{C}e^{L(n-q)}o(e^{L(q-n)}).
		\end{aligned}
	\end{equation}
	Hence, the proposition follows by Theorem \ref{thm:AH_mass_top_face}.
\end{proof}

Combining the equality case of the positive mass theorem, we deduce that a nontrivial asymptotically hyperbolic metric (with $R_g\ge -n(n-1)$) cannot be localized in a region $U$ satisfying \eqref{eqn:theta}. One can interpret $\Theta(U)$ as \emph{the asymptotic size of a subset} $U$ at infinity. Note that a related concept was considered for the asymptotically flat setting by Carlotto and Schoen (see \cite[Section 2.2]{Carlotto:2016dh}).

For a conformally compact asymptotically hyperbolic manifold, the value of $\Theta(U)$ is equivalent to the measure of $\partial_\infty U$, where $\partial_\infty U$ is the boundary of $U$ intersecting with the conformal infinity. Since the mass of such manifolds can be defined as the integral on the conformal infinity (see \cite{Wang.X:2001}), it is natural to expect that any region $U$ with $\Theta(U)=0$ does not affect the total mass.

We also remark that an asymptotically hyperbolic metric is localizable in a subset $U$ with $\Theta(U)>0$. In other words, for a region $U$ with $\Theta(U)>0$ in $\mathbb{H}^n$, there exists an asymptotically hyperbolic metric $(\mathbb{H}^n,g)$ with $R_g\ge -n(n-1)$ such that $g$ is isometric to $b$ outside $U$ and $R_g>-n(n-1)$ somewhere in $U$, proved by Chru\'sciel and Delay \cite{Chrusciel:2018aa}.

%Another ingredient of the proof of Theorem \ref{thm:AH_mass_horospheres} is the relation between the mass $1$-form and the mean curvature difference (see Proposition \ref{prop:1form}). Once we observed it, the remaining proof consists of direct estimates of the integrals on the boundary of large parabolic cylinders $C_L$. 

The rest of the paper is organized as follows: in section \ref{sec:preliminaries}, we analyze the mass integrand and discover its relation to the mean curvature difference, see \eqref{eqn:1form_formula}. Also, we verify Theorem \ref{thm:AH_mass_horospheres} for AdS Schwarzschild metrics as an example. In section \ref{sec:proof}, we prove the main theorems. %We then introduce a way to characterize non-localizable regions by using the formula \eqref{eqn:AH_mass_horospheres}. 

If applied to other hypersurfaces, the formula \eqref{eqn:1form_formula} will yield other mass formula.
%Our method also leads to other mass formulas in terms of the difference of the mean curvatures and the areas. %of coordinate spheres with respect to the given metric and the background metric. 
While we do not use them to prove the main results, we include some of those formulas in an appendix. 

%\bigskip
%
%\noindent {\bf Acknowledgements.}
%The authors would like to thank the organizers of the 2020 Virtual Workshop on Ricci and Scalar Curvature, which provided a lot of inspiring talks. Especially, the plenary talk of Chao Li influenced the initial idea of this work. H. C. Jang was partially supported by NSF Grant DMS-1612049 for participating in the workshop and would like to thank Christina Sormani for funding. P. Miao acknowledges the support of NSF Grant DMS-1906423. The authors are also grateful to Erwann Delay for his interest and insightful comments. 

\section{Preliminaries}\label{sec:preliminaries}
We recall various coordinates of hyperbolic space $(\mathbb{H}^n,b)$.

\vspace{.1cm}

\begin{enumerate}
	\item Hyperbolic space can be obtained as the upper sheet of the hyperboloid in the Minkowski space:
	\[
		\mathbb{H}^n=\{(z,t)\in\mathbb{R}^{n,1}:z_1^2+\cdots+z_n^2-t^2=-1, t>0\}.
	\]
	The set of functions $\{z_1,\ldots,z_n\}$ restricted to $\mathbb{H}^n$ are often called the hyperboloidal coordinates. Let $r=\sqrt{z_1^2+\cdots+z_n^2}$, then the metric $b$ is written as
	\[
		b=\frac{dr^2}{1+r^2}+r^2g_{\mathbb{S}^{n-1}}.
	\]
	On $\mathbb{H}^n$, it satisfies that $t=\sqrt{1+r^2}$, so we denote a function $t=\sqrt{1+r^2}$ throughout this paper.
	\item Consider the upper half space model $\mathbb{R}^n_+=\{y\in\mathbb{R}^n:y_1>0\}$. Then the metric $b$ is
	\[
		b=\frac{1}{y_1^2}(dy_1^2+\cdots +dy_n^2).
	\]
	\item Define $x_1=-\ln y_1,x_2=y_2,\ldots, x_n=y_n$. Then $\{x_1,\ldots,x_n\}$ becomes another set of coordinates of $\mathbb{H}^n$, and the metric $b$ is 
	\[
		b=dx_1^2+e^{2x_1}(dx_2^2+\cdots +dx_n^2).
	\]
\end{enumerate}
These coordinates are related by 
%between (1) and (2) is
%\begin{equation}\label{eqn:yz-transform}
%	y_1=\frac{1}{t-z_1},\quad y_i=\frac{z_i}{t-z_1}\text{ for }i=2,\ldots,n,
%\end{equation}
%so we get
\begin{equation}\label{eqn:xz-transform}
y_1=\frac{1}{t-z_1},\quad y_i=\frac{z_i}{t-z_1}, \ 
	e^{x_1}=t-z_1,\quad x_i=\frac{z_i}{e^{x_1}},  \ i=2,\ldots,n.
\end{equation}
Using these transforms, we can write a horosphere $\mathcal{H}_L$ as
\[
	\mathcal{H}_L=\{x_1=L\}=\{y_1=e^{-L}\}=\{t-z_1=e^L\}.
\]
For simplicity, we adopt the notation $\hat{x}=(x_2,\ldots,x_n)$. The same applies to $\hat{y}$ and $\hat{z}$.
In section \ref{sec:proof}, we use the coordinates (3) to get essential estimates.

\subsection{The mass $1$-form and the mean curvature difference}

It is well-known that the integrands of the mass integral \eqref{eqn:AH mass integral} can be viewed as the following $1$-form:
\begin{equation}\label{eqn:mass_1_form}
	\mathbb{U}(V)=V\,\mathrm{div }_b h-V\,d(\mathrm{tr}_{b} h)+(\mathrm{tr}_{b} h)d V-h(\mathring{\nabla} V,\cdot).
\end{equation}
Using this one form and the linearity of the mass integral $H_\Phi$, we can write
\[
	p_0-p_1=\lim_{r\to\infty}\int_{S_r}\mathbb{U}(t-z_1)(\nu_0)\, d\sigma_b.
\]

To derive the mass formula in terms of geometric quantity, we use the following proposition showing another expression of the $1$-form $\mathbb{U}$. Note that the following proposition holds for the $1$-form $\mathbb{U}$ with an arbitrary background metric $g_0$ replacing $b$.

\begin{proposition}\label{prop:1form}
	Let $\Sigma$ be any hypersurface in $M$ and $g_0$, $g$ be Riemannian metrics on $M$. Let $\nu_0,\nu_g$ denote the unit normal vectors to $\Sigma$ pointing the same side in $(M,g_0),(M,g)$, respectively. Then for $|h|_{g_0}$ sufficiently small (where $h=g-g_0$), we have
	\begin{equation}\label{eqn:1form_formula}
		\begin{aligned}
			\mathbb{U}(V)(\nu_0)=V[2(H_{g_0}-H_g)+&|A_{g_0}|_{g_0}O(|h|_{g_0}^2)+O(|\mathring{\nabla}h|_{g_0}|h|_{g_0})]\\
			&+(\mathrm{tr}_{g_0}^\Sigma h)dV(\nu_0)-V\langle A_{g_0},h\rangle_{g_0|\Sigma}-\mathrm{div}_\Sigma (VX).
		\end{aligned}
	\end{equation}
	where $H_{g_0}, H_g$ are the mean curvature of $\Sigma$ with respect to $\nu_0, \nu_g$ in $(M,g_0), (M,g)$, respectively, $A_{g_0}$ is the second fundamental form of $\Sigma$ with respect to $\nu_0$ in $(M,g_0)$, and $X$ is the vector field on $\Sigma$ that is dual to the $1$-form $h(\nu_0,\cdot)$ with respect to $g_0|_\Sigma$.
\end{proposition}

\begin{proof}
	By \cite[Proposition 2.1]{Miao.2021}, we have
	\[
		\begin{aligned}
			\mathbb{U}(V)(\nu_0)&=V[2(H_{g_0}-H_g)-\mathrm{div}_\Sigma X-\langle A_{g_0},h\rangle_{g_0|_\Sigma}+|A_{g_0}|_{g_0}O(|h|_{g_0}^2)\\
			&\hspace{1.5in}+O(|\mathring{\nabla}h|_{g_0}|h|_{g_0})]+(\mathrm{tr}_{g_0} h)dV(\nu_0)-h(\mathring{\nabla}V,\nu_0).
		\end{aligned}
	\] Then we get
	\[
		\begin{aligned}
			\mathbb{U}(V)(\nu_0)&=V[2(H_{g_0}-H_g)-\mathrm{div}_\Sigma X-\langle A_{g_0},h\rangle_{g_0|_\Sigma}+|A_{g_0}|_{g_0}O(|h|_{g_0}^2)+O(|\mathring{\nabla}h|_{g_0}|h|_{g_0})]\\
			&\hspace{0.5in}+((\mathrm{tr}_{g_0}^\Sigma h)dV(\nu_0)+h(\nu_0,\nu_0)dV(\nu_0))-(h(\mathring{\nabla}^\Sigma V,\nu_0)+h(dV(\nu_0)\nu_0,\nu_0))\\
			&=V[2(H_{g_0}-H_g)-\langle A_{g_0},h\rangle_{g_0|_\Sigma}+|A_{g_0}|_{g_0}O(|h|_{g_0}^2)+O(|\mathring{\nabla}h|_{g_0}|h|_{g_0})]\\
			&\hspace{0.5in}+(\mathrm{tr}_{g_0}^\Sigma h)dV(\nu_0)-(\mathrm{div}_\Sigma X+h(\mathring{\nabla}^\Sigma V,\nu_0))\\
			&=V[2(H_{g_0}-H_g)+|A_{g_0}|_{g_0}O(|h|_{g_0}^2)+O(|\mathring{\nabla}h|_{g_0}|h|_{g_0})]\\
			&\hspace{0.5in}+(\mathrm{tr}_{g_0}^\Sigma h)dV(\nu_0)-V\langle A_{g_0},h\rangle_{g_0|\Sigma}-\mathrm{div}_\Sigma (VX).
		\end{aligned}
	\]
\end{proof}

\begin{remark}\label{rmk:mass exhaustion}
	The mass integrals can be computed with a suitable exhaustion of $M$ consisting of bounded domains to which the divergence theorem is applicable. See \cite[Proposition 4.1]{Bartnik:1986dq} and \cite[Section 2]{Michel:2011jz}.
\end{remark}

\subsection{Example case: Anti-de Sitter Schwarzschild manifolds}

Here, we compute the mass of the anti-de Sitter (AdS) Schwarzschild manifolds by using the formula \eqref{eqn:AH_mass_horospheres}. For a given $m>0$, we call a Riemannian manifold $[r_m,\infty)\times S^{n-1}$ equipped with the metric
\[
	g_m=\frac{dr^2}{1+r^2-\frac{2m}{r^{n-2}}}+r^2 g_{\mathbb{S}^{n-1}}
\]
the AdS Schwarzschild manifold where $r_m$ is the largest zero of $r^n+r^{n-2}-2m$. By setting $z_1=r\cos\theta$ for $\theta\in (0,\pi)$, we can write the metric as
\[
	g_m=\frac{dr^2}{1+r^2-\frac{2m}{r^{n-2}}}+r^2(d\theta^2+(\sin\theta)^2 g_{\mathbb{S}^{n-2}}).
\]
Let $\xi=-\frac{2m}{r^{n-2}}$ and $V=t-z_1=\sqrt{1+r^2}-r\cos\theta$. By direct computation, we have 
\begin{equation}
	\begin{aligned}
		\nabla V&=\mathring{\nabla}V+\xi V_r\partial_r,\\
		|\nabla V|_g^2&=|\mathring{\nabla}V|_b^2+\xi (V_r)^2,\\
		\Delta_g V&=\Delta_b V-V_r\left(\frac{n\xi}{2r}\right)+\xi V_{rr},\\
		\nabla_\nu\nabla_\nu V&=\mathring{\nabla}_{\nu_0} \mathring{\nabla}_{\nu_0} V+(2-n)\frac{r(V_r)^3 \xi}{2V^2}+O(r^{-(n+1)}),
	\end{aligned}
\end{equation}
where $\nabla, \Delta_g, \nu$ are the gradient, Laplacian, normal vector to the level set $\{V=e^L\}$ pointing the direction $\nabla V$ with respect to the metric $g$, while $\mathring{\nabla}, \Delta_b, \nu_0$ are the corresponding ones with respect to the hyperbolic metric $b$. Also, the subscript $r$ on $V$ means the partial differentiation with respect to $r$.

We observe that on $\{V=e^L\}$
%\[
%	\begin{aligned}
$		e^{L}=\sqrt{1+r^2}-r\cos\theta\le \sqrt{1+r^2}+r $,
%	\end{aligned}
%\]
thus it follows that $r^{-1}=O(e^{-L})$ when $L$ is large.

By using these, we compute the mean curvature of the level set $\{V=e^L\}$ as $L$ approaches infinity:
\begin{equation}
	\begin{aligned}
		H_g&=\frac{\Delta_g V-\nabla_\nu\nabla_\nu V}{|\nabla V|_g}\\
		&=\left(\Delta_b V-\mathring{\nabla}_{\nu_0}\mathring{\nabla}_{\nu_0}V+V_r\left(\frac{n\xi}{2r}\right)-\xi V_{rr}-(2-n)\frac{r(V_r)^3 \xi }{V^2}+O(e^{-(n+1)L})\right)\\
		&\hspace{1.7in}\times\frac{1}{|\mathring{\nabla} V|_b}\left(1-\frac{\xi (V_r)^2}{2|\mathring{\nabla} V|_b^2}+O(e^{-(2n)L})\right),\\
		&=H_b-H_b\frac{\xi (V_r)^2}{2V^2}+\frac{1}{V}\left(V_r\left(\frac{n\xi }{2r}\right)-(2-n)\frac{r(V_r)^3 \xi }{2V^2}\right)+O(e^{-(n+2)L}).
	\end{aligned}
\end{equation}
Here, we used the fact that $V=|\mathring{\nabla}V|_b=\mathring{\nabla}_{\nu_0}\mathring{\nabla}_{\nu_0} V$ and $V_{rr}=\frac{1}{(\sqrt{1+r^2})^3}=O(r^{-3})$. 

Therefore, we have
\begin{equation}\label{eqn:AdS mean curv difference}
	\begin{aligned}
		2V(H_b-H_g)&=(n-1)\frac{\xi (V_r)^2}{V}-V_r\left(\frac{n\xi }{r}\right)+(2-n)\frac{r(V_r)^3 \xi }{V}+O(e^{-(n+1)L})\\
		&=\frac{\xi  V_r}{r}\left(1-n+O(e^{-L})+O(e^{-2L})\right)\\
		&=\frac{2m}{r^{n-1}}\left(n-1-(n-1)\cos\theta+O(e^{-L})\right)\\
		&=\frac{2m}{r^{n-1}}\left(n-1-(n-1)\frac{z_1}{r}+O(e^{-L})\right).
	\end{aligned}
\end{equation}

Hence, we obtain
\begin{equation}
	\begin{aligned}
		&\int_{\mathcal{H}_L}2V(H_b-H_g)\, d\mu_g\\
		&\qquad=\int_{\mathcal{H}_L}\frac{2m}{r^{n-1}}\left(n-1-(n-1)\frac{z_1}{r}+O(e^{-L})\right)\, d\mu_b+o(1)\\
		&\qquad=\int_{\mathbb{S}^{n-2}}\int_0^\infty\frac{2m}{r^{n-1}}\left(n-1-(n-1)\frac{z_1}{r}\right) e^{L(n-1)}\rho^{n-2}\, d\rho d\sigma_{\mathbb{S}^{n-2}}+o(1)\\
		&\qquad=2m(n-1)\omega_{n-1}+o(1)
	\end{aligned}
\end{equation}
as $L$ approaches infinity. In the second equality, we used the spherical coordinates on $\mathcal{H}_L$ so that
\[
	b|_{\mathcal{H}_L}=e^{2L}(d\rho^2+\rho^2 g|_{\mathbb{S}^{n-2}}).
\]
	
\begin{remark}
	From \eqref{eqn:AdS mean curv difference}, the difference $H_b-H_g$ is positive for sufficiently large $r>R$. Indeed, by using
	\[
		\cos\theta=\frac{\sqrt{1+r^2}-e^L}{r}\quad \text{on }\mathcal{H}_L,
	\]
	we have for large $L$
	\[
		\begin{aligned}
			1-\cos\theta &=\frac{r-\sqrt{1+r^2}+e^L}{r}=\frac{1}{r}\left(e^L-\frac{1}{\sqrt{1+r^2}+r}\right)> 0.
		\end{aligned}
	\]
	This implies that for the AdS Schwarzschild manifold, the mean curvature on coordinate horospheres $\mathcal{H}_L$ for large $L$ is less than $n-1$.
\end{remark}

\section{Hyperbolic mass via horospheres}\label{sec:proof}

In this section, we derive the mass formula in terms of geometric quantities on coordinate horospheres. Assume $(M^n,g), n\ge 3$, is asymptotically hyperbolic. By using the $\{x_i\}$ coordinates defined as (3) in section \ref{sec:preliminaries}, the parabolic cylinder $C_L$ is defined as
\begin{equation}
	C_L=\{(x_1,\hat{x})\in \mathbb{H}^{n}:|x_1|\le L, |\hat{x}|\le \sigma(L)\}
\end{equation}
where $\sigma(L)$ increases to infinity as $L\to\infty$, which will be determined later. Define the surrounding surfaces of $C_L$ and their boundaries as the following:
\[
	\begin{aligned}
		F_{\pm,L}&=\{(x_1,\hat{x})\in\mathbb{H}^{n}:x_1=\pm L,|\hat{x}|\le \sigma(L)\},\\
		S_L&=\{(x_1,\hat{x})\in\mathbb{H}^{n}:|x_1|\le L,|\hat{x}|=\sigma(L)\}\\
		E_{\pm,L}&=\{(x_1,\hat{x})\in\mathbb{H}^{n}:x_1=\pm L,|\hat{x}|=\sigma(L)\}
	\end{aligned}
\]
See Figure \ref{fig:prism} to compare the shapes of $C_L$ in various coordinate charts. For convenience, all estimates throughout this section are performed in the $\{x_i\}$-coordinates.

\begin{figure}
	\centering
	\begin{subfigure}[b]{0.4\textwidth}
		\centering
		\fontsize{11pt}{11pt}\selectfont
		\def\svgwidth{2.5in}
		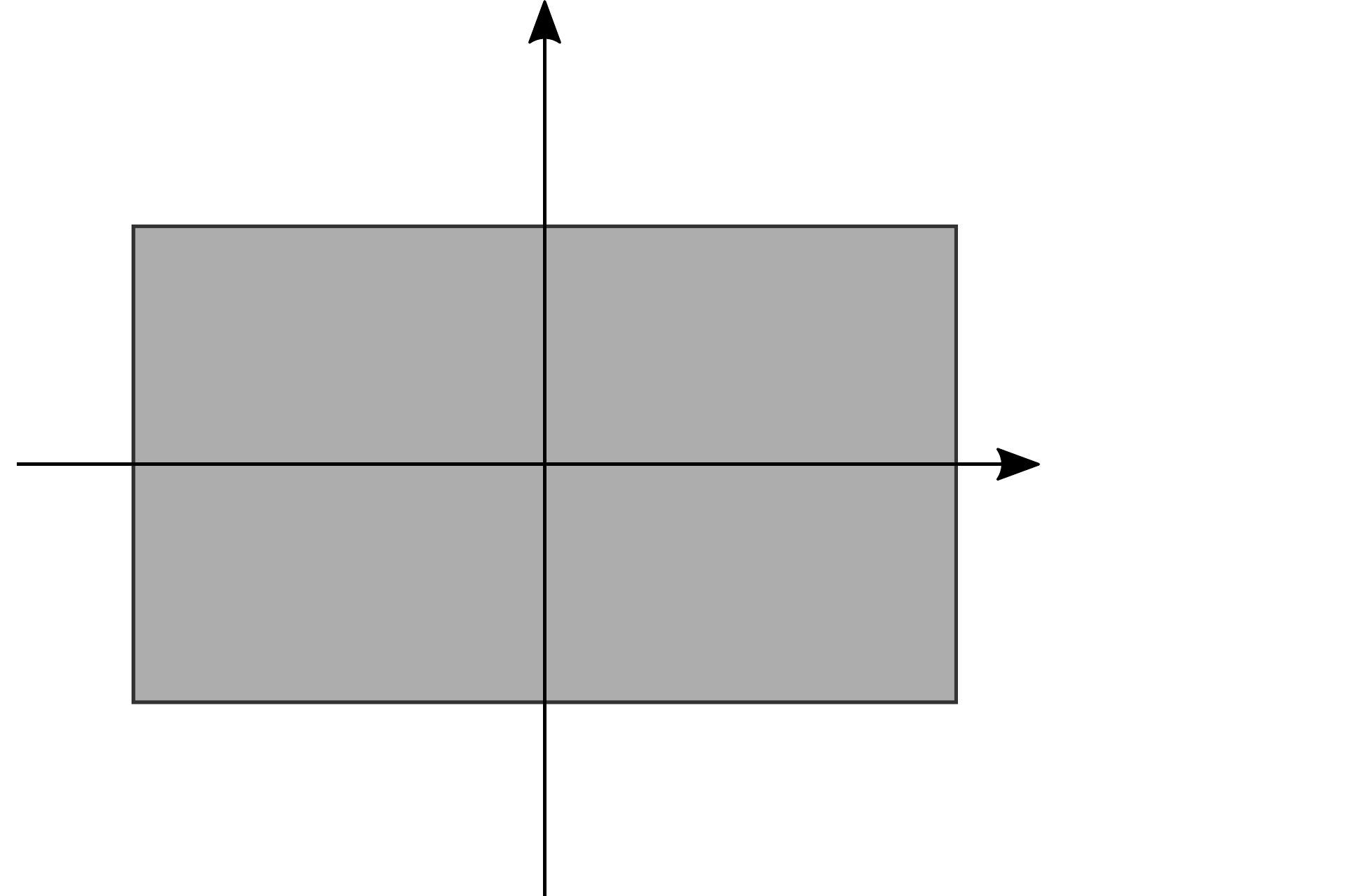
		\caption{In $\{x_i\}$-coordinates}
	\end{subfigure}
	\qquad
	\begin{subfigure}[b]{0.4\textwidth}
		\centering
		\fontsize{11pt}{11pt}\selectfont
		\def\svgwidth{2.5in}
		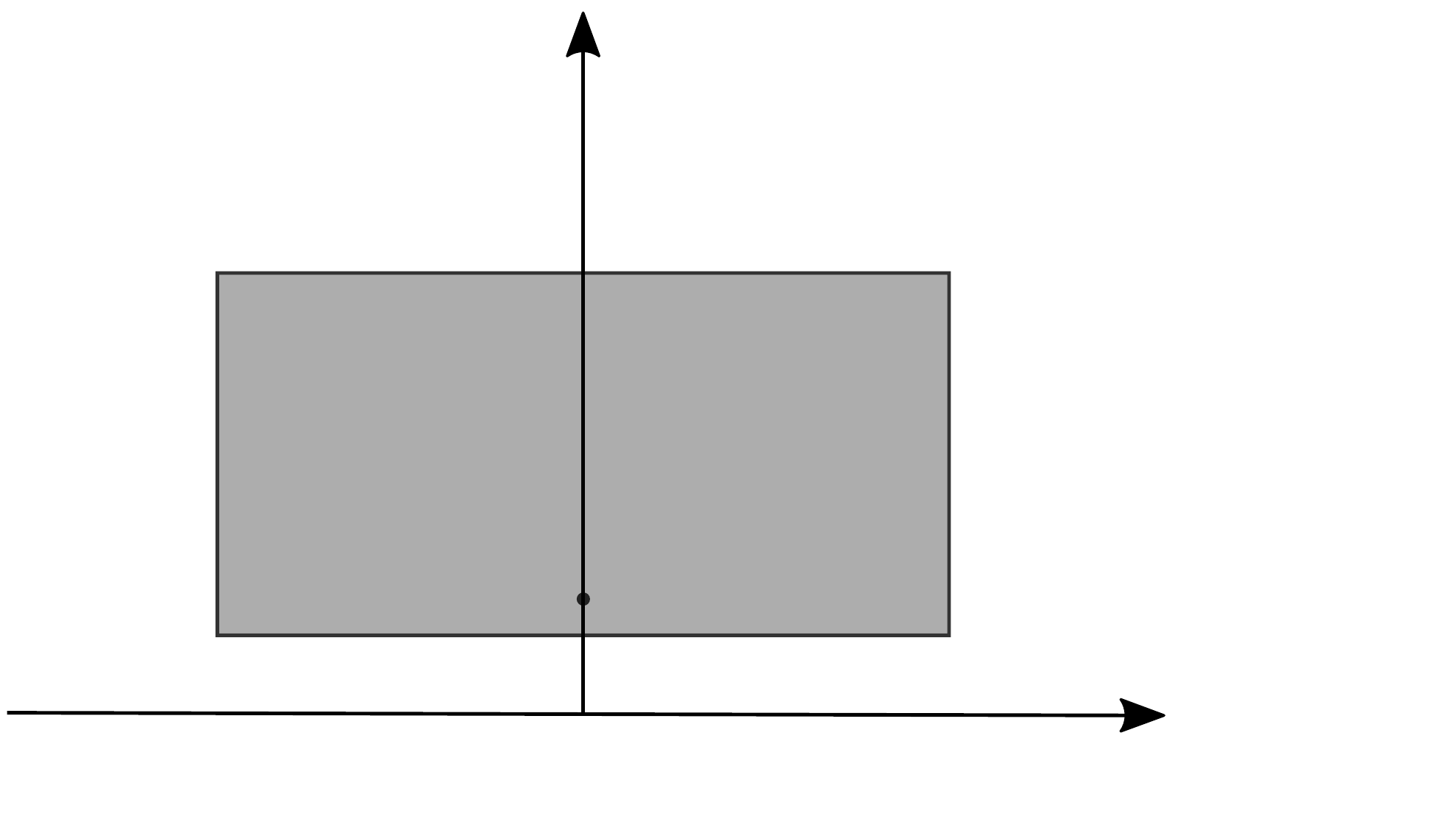
		\caption{Upper half space model}
	\end{subfigure}
	\hfill\\
	\hfill\\
	\begin{subfigure}[b]{0.4\textwidth}
		\centering
		\fontsize{11pt}{11pt}\selectfont
		\def\svgwidth{2.5in}
		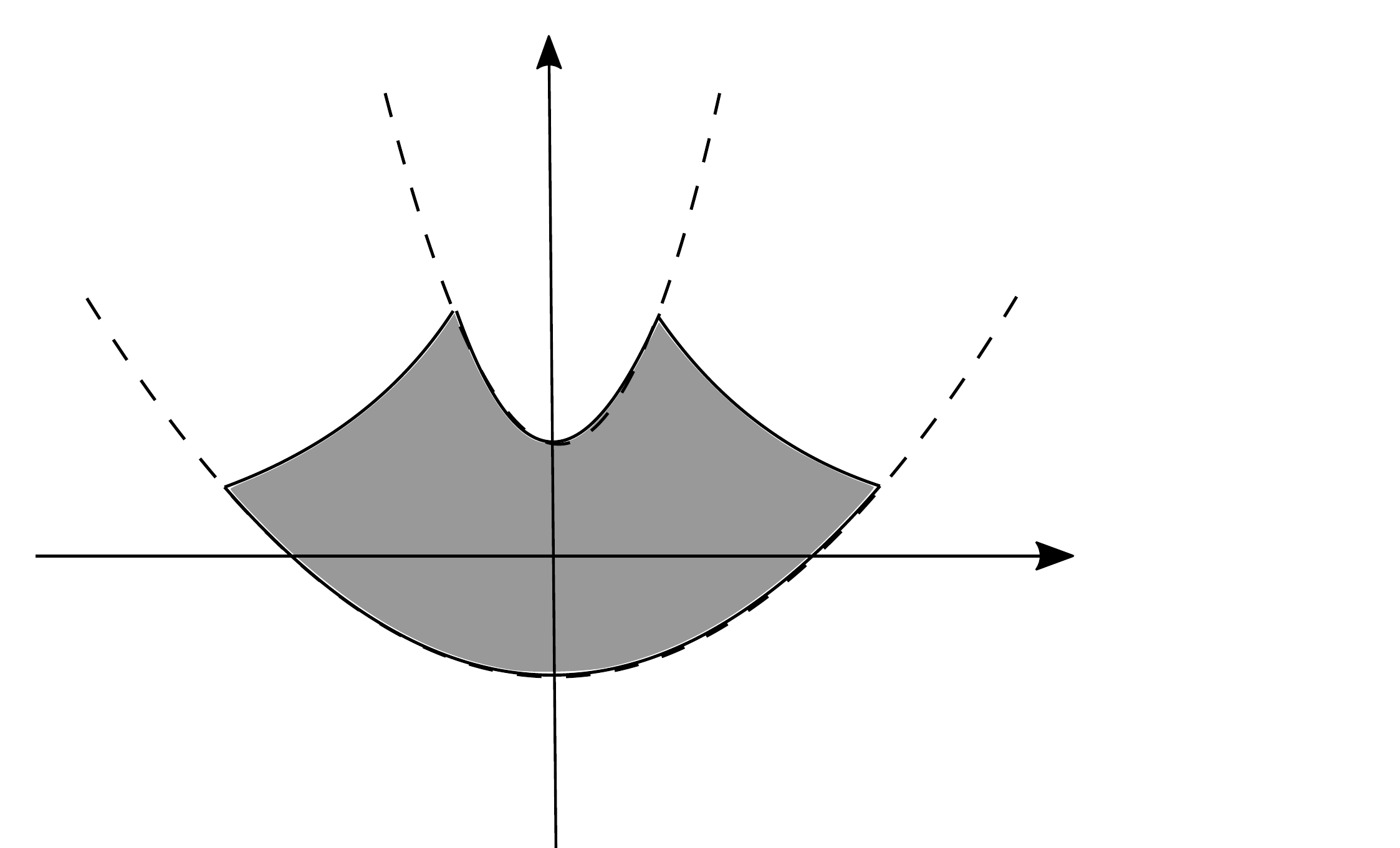
		\caption{Hyperboloidal model}
	\end{subfigure}
	\qquad
	\begin{subfigure}[b]{0.4\textwidth}
		\centering
		\fontsize{11pt}{11pt}\selectfont
		\def\svgwidth{2.3in}
		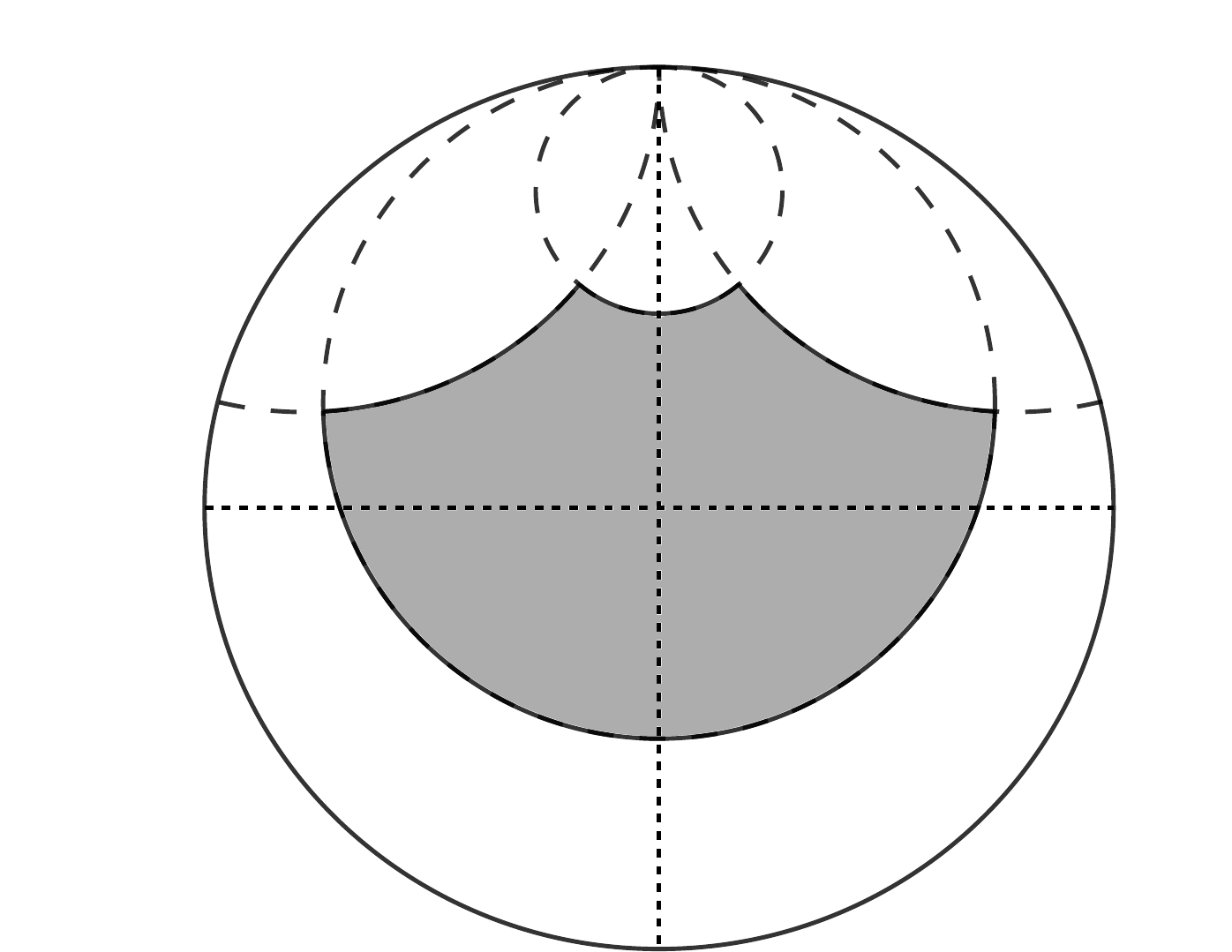
		\caption{Poincar\'e Ball model}
	\end{subfigure}
	\caption{$C_L$ in various coordinate charts}
	\label{fig:prism}
\end{figure}

We first prove Theorem \ref{thm:AH_mass_top_face} by using the following formula (see Remark \ref{rmk:mass exhaustion})
\[
	p_0-p_1=\lim_{L\to\infty}\int_{\partial C_L}\mathbb{U}(t-z_1)(\nu_0)\, d\sigma_b.
\]

Let $V=t-z_1=e^{x_1}$. We write the one form $\mathbb{U}$ on the surrounding surfaces using Lemma \ref{prop:1form}.

\begin{lemma}
	In the above setting, the following hold:
	\begin{equation}\label{eqn:flux1}
		\begin{aligned}
			\int_{F_{\pm,L}}\mathbb{U}(V)(\nu_0)\, d\sigma_b&=\int_{F_{\pm,L}}2V(H_b-H_g)\, d\sigma_g\\
			&\qquad +\int_{F_{\pm,L}} VO(|h|_b^2)\, d\sigma_b+\int_{E_{\pm,L}}VO(|h|_b)\, ds_b,
		\end{aligned}
	\end{equation}
	and
	\begin{equation}\label{eqn:flux2}
		\begin{aligned}
			\int_{S_{L}}\mathbb{U}(V)(\nu_0)\, d\sigma_b&=\int_{S_{L}}2V(H_b-H_g)\, d\sigma_g\\
			&\qquad +\int_{S_{L}} \left[VO(|h|_b^2)+\sigma(L)^{-1}O(|h|_b)\right]\, d\sigma_b+\int_{E_{\pm,L}}VO(|h|_b)\, ds_b,
		\end{aligned}
	\end{equation}
	as $L\to\infty$.
\end{lemma}

\begin{proof}

By direct computation, we have
\[
	V=|\mathring{\nabla}V|_b.
\]
Since $F_{\pm,L}\subset\{V=e^{\pm L}\}$ are level sets of $V$ from \eqref{eqn:xz-transform}, it follows that
\[
	(\mathrm{tr}_{b}^{F_{\pm,L}} h)dV(\nu_0)-V\langle A_{b},h\rangle_{b|_{F_{\pm,L}}}=\pm(\mathrm{tr}_{b}^{F_{\pm,L}} h)(V-V)=0.
\]
By integrating it, we get
\[
	\begin{aligned}
		&\int_{F_{\pm,L}}\mathbb{U}(V)(\nu_0)\, d\sigma_b \\
		&=\int_{F_{\pm,L}}2V(H_{b}-H_g)\, d\sigma_g+\int_{F_{\pm,L}}VO(|h|^2_b)\, d\sigma_b+\int_{E_{\pm,L}} VO(|h|_b)\,ds_b,
	\end{aligned}
\]
which proves \eqref{eqn:flux1}.

Now we consider the lateral surface $S_L$. By writing the metric
\[
	b=dx_1^2+e^{2x_1}\sum_{i=2}^{n}dx_i^2=dx_1^2+e^{2x_1}(d\rho^2+\rho^2 g_{\mathbb{S}^{n-2}}),
\]
we have
\[
	\Gamma_{1\rho}^{1}=\Gamma_{1\rho}^{\alpha}=\Gamma_{\alpha\rho}^1=0,\Gamma_{\alpha\rho}^\beta=\frac{1}{2}b^{\beta\tau}(b_{\alpha\tau,\rho})=\rho^{-1}\delta_\alpha^\beta,
\]
where the indices $\alpha, \beta$ represent local orthonormal coordinates on $(\mathbb{S}^{n-2},g_{\mathbb{S}^{n-2}})$. Thus, the second fundamental form of $S_L$ in hyperbolic space is
\[
	A_{11}=A_{1\alpha}=A_{\alpha 1}=0, A_{\alpha\beta}=e^{-x_1}\Gamma_{\alpha \rho}^{\tau}b_{\tau\beta}=e^{x_1}\sigma(L)\delta_{\alpha\beta}.
\]
It follows that
\[
	V\langle A_b,h\rangle_{b|_{S_L}}=e^{x_1}(b^{\alpha\beta}b^{\tau\sigma} A_{\alpha\tau}h_{\beta\sigma})\le (n-2)\sigma(L)^{-1}|h|_b,
\]
so by using Proposition \ref{prop:1form}, we get on $S_L$
\[
	\begin{aligned}
		\mathbb{U}(V)(\nu_0)&=V[2(H_{b}-H_g)+|A_{b}|_{b}O(|h|_{b}^2)+O(|\mathring{\nabla}h|_{b}|h|_{b})]\\
		&\hspace{0.9in}+(\mathrm{tr}_{b}^{S_L} h)dV(\nu_0)-V\langle A_{b},h\rangle_{b|_{S_L}}-\mathrm{div}_{S_L} (VX)\\
		&=V[2(H_{b}-H_g)+O(|h|^2_b)]+\sigma(L)^{-1}O(|h|_b)-\mathrm{div}_{S_L}(VX).
	\end{aligned}
\]
Hence, we have
\[
	\begin{aligned}
		&\int_{S_{L}}\mathbb{U}(V)(\nu_0)\, d\sigma_b \\
		&=\int_{S_{L}}2V(H_{b}-H_g)\, d\sigma_g+\int_{S_{L}}\left[VO(|h|_b^2)+\sigma(L)^{-1}O(|h|_b)\right]\, d\sigma_b+\int_{E_{\pm,L}} VO(|h|_b)\,ds_b.
	\end{aligned}
\]

\end{proof}

In what follows, we present the essential estimates for each surface and edge. We use that the error tensor $h$ decays as $|h|_b=O(r^{-q})$ and the following formula
\begin{equation}\label{eqn:z-radial}
	r^2=\left(\cosh x_1+e^{x_1}\frac{|\hat{x}|^2}{2}\right)^2-1.
\end{equation}
\\

$\bullet$ On $F_{+,L}$:
\begin{equation}\label{eqn:estimate_top_face}
	\begin{aligned}
		&\int_{F_{+,L}}VO(|h|^2_b)\, d\sigma_b\\
		&\qquad\le C_1\omega_{n-2}\int_{0}^{\sigma(L)} e^L\left(\left(\cosh L+\frac{e^{L}\rho^2}{2}\right)^2-1\right)^{-q}e^{L(n-1)}\rho^{n-2}\,d\rho\\
		&\qquad\le C_2e^{L(n-2q)}\int_0^{\infty}\left(1+e^{-2L}+\rho^2\right)^{-2q}\rho^{n-2}\, d\rho\\
		&\qquad\le C_3e^{L(n-2q)}\int_{0}^{\infty}(1+e^{-2L}+\rho^2)^{-2q+\frac{n-3}{2}}\rho^{n-2-n+3}\, d\rho\\
		&\qquad\le \tilde{C}e^{L(n-2q)}.
	\end{aligned}
\end{equation}
Here, we need the condition $1-2q+\frac{n-3}{2}<0$, i.e., $q>\frac{n-1}{4}$ to get the last inequality. Note that $\sigma(L)$ does not affect this estimate.\\

$\bullet$ On $F_{-,L}$:\\
We estimate with the integrand $VO(|h|_b)$ to show that the term $V(H_b-H_g)$ on $F_{-,L}$ has no contribution to the mass.
\begin{equation}\label{eqn:estimate_bottom_face}	
	\begin{aligned}
		&\int_{F_{-,L}}VO(|h|_b)\, d\sigma_b\\
		&\qquad\le C_1\omega_{n-2}\int_{0}^{\sigma(L)} e^{-L}\left(\left(\cosh L+\frac{e^{-L}\rho^2}{2}\right)^2-1\right)^{-\frac{q}{2}}e^{-L(n-1)}\rho^{n-2}\,d\rho\\
		&\qquad\le C_2\cdot 4^q e^{L(-n-q)}\int_0^{\infty}\left(1+e^{-2L}+\rho^2e^{-2L}\right)^{-q}\rho^{n-2}\, d\rho\\
		&\qquad\le C_3 e^{L(-n-q)}\int_{0}^{\infty}(1+e^{-2L}+\rho^2e^{-2L})^{-q+\frac{n-3}{2}}\rho^{n-2-n+3}e^{L(n-3)}\, d\rho\\
		&\qquad\le \tilde{C}e^{L(-n-q)}e^{L(n-1)}=\tilde{C}e^{L(-1-q)}.
	\end{aligned}
\end{equation}
Similarly, we require the condition $1-2q+\frac{n-1}{2}<0$, i.e., $q>\frac{n+1}{4}$ to get the last inequality. Again, $\sigma(L)$ does not affect this estimate.\\

$\bullet$ On $E_{+,L}$:
\begin{equation}\label{eqn:estimate_top_edge}
	\begin{aligned}
		\int_{E_{+,L}}VO(|h|_b)\, d\sigma_b&\le C_1\omega_{n-2}\cdot e^{L}\left(\left(\cosh L+\frac{e^{L}\sigma(L)^2}{2}\right)^2-1\right)^{-\frac{q}{2}}e^{L(n-2)}\sigma(L)^{n-2}\\
		&\le C_2e^{L(n-1)}\sigma(L)^{n-2}\left(\frac{2}{e^{L}+e^{-L}+e^{L}\sigma(L)^2}\right)^{q}\\
		&\le C_3e^{L(n-1)}\sigma(L)^{n-2}e^{-qL}\sigma(L)^{-2q}\\
		&\le\tilde{C}e^{L(n-1-q)}\sigma(L)^{n-2-2q}.
	\end{aligned}
\end{equation}
Thus, to make the above integral vanish as $L\to\infty$, we need the following condition for $\sigma(L)$:
\begin{equation}\label{eqn:k_condition_1}
	\sigma(L)^{n-2-2q}=o(e^{L(q-n+1)})\text{ as }L\to\infty.
\end{equation}

Note that if the falloff rate $q\ge n-1$, $\sigma(L)$ being increasing to infinity is sufficient to show the above integral vanishes as $L\to\infty$.\\

$\bullet$ On $E_{-,L}$:
\begin{equation}\label{eqn:estimate_bottom_edge}
	\begin{aligned}
		&\int_{E_{-,L}}VO(|h|_b)\, d\sigma_b\\
		&\qquad\le C_1\omega_{n-2}\cdot e^{-L}\left(\left(\cosh L+\frac{e^{-L}\sigma(L)^2}{2}\right)^2-1\right)^{-\frac{q}{2}}e^{-L(n-2)}\sigma(L)^{n-2}\\
		&\qquad\le C_2e^{-L(n-1)}\sigma(L)^{n-2}\left(\frac{2}{e^{L}+e^{-L}+e^{-L}\sigma(L)^2}\right)^{q-\frac{n}{2}}\left(\frac{2}{e^{L}+e^{-L}+e^{-L}\sigma(L)^2}\right)^{\frac{n}{2}}\\
		&\qquad\le C_2e^{-L(n-1)}\sigma(L)^{n-2}o(1)\left(\frac{2}{e^{L}+e^{-L}+e^{-L}\sigma(L)^2}\right)^{\frac{n}{2}}\\
		&\qquad\le C_3e^{-L(n-1)}\sigma(L)^{n-2}o(1)e^{\frac{n}{2}L}\sigma(L)^{-n}\\
		&\qquad\le\tilde{C}e^{-\frac{n-2}{2}L}\sigma(L)^{-2}o(1).
	\end{aligned}
\end{equation}
The third inequality follows from the condition $q>\frac{n}{2}$. Hence, it follows that the above integral vanishes as $L\to\infty$ regardless of $\sigma(L)$.\\

$\bullet$ On $S_L$:\\
Similar to the case on $F_{-,L}$, we first estimate with the integrand $VO(|h|_b)$.
\begin{equation}\label{eqn:estimate_lateral_surface_1}
	\begin{aligned}
		&\int_{S_L}VO(|h|_b)\, d\sigma_b\\
		&\qquad\le C_1\omega_{n-2}\int_{-L}^{L}e^{x_1}\left(\left(\cosh x_1 +\frac{e^{x_1}\sigma(L)^2}{2}\right)^2-1\right)^{-\frac{q}{2}}e^{x_1(n-2)}\sigma(L)^{n-2}\, dx_1\\
		%&\le C_2\omega_{n-2}\int_{-L}^{L}e^{x_1}\left(\cosh x_1 +\frac{e^{x_1}e^{nL}}{2}\right)^{-q}e^{x_1(n-2)}e^{\frac{n(n-2)}{2}L}\, dx_1\\
		&\qquad\le C_2\int_{-L}^L e^{x_1(n-1)}\sigma(L)^{n-2}\left(\frac{2}{e^{x_1}+e^{-x_1}+e^{x_1}\sigma(L)^2}\right)^q\, dx_1.
	\end{aligned}
\end{equation}
We split the last integral into two parts: first, we have
\begin{equation}\label{eqn:estimate_lateral_surface_2}
	\begin{aligned}
		&\int_{0}^L e^{x_1(n-1)}\sigma(L)^{n-2}\left(\frac{2}{e^{x_1}+e^{-x_1}+e^{x_1}\sigma(L)^2}\right)^q\, dx_1\\
		&\le C\int_{0}^L e^{x_1(n-1-q)}\sigma(L)^{n-2-2q}\, dx_1\le
		\left\{\begin{array}{l}
			\tilde{C}_1 \sigma(L)^{n-2-2q}(e^{L(n-1-q)}+1),\vspace{3pt}\\
			\tilde{C}_2 \sigma(L)^{n-2-2q}(L^{-1}+1)
		\end{array}\right.
	\end{aligned}
\end{equation}
Hence, we need an additional condition to make this integral converge to zero, which turns out to be the same as \eqref{eqn:sigma_condition}.

Next, we have
\begin{equation}\label{eqn:estimate_lateral_surface_3}
	\begin{aligned}
		&\int_{-L}^0 e^{x_1(n-1)}\sigma(L)^{n-2}\left(\frac{2}{e^{x_1}+e^{-x_1}+e^{x_1}\sigma(L)^2}\right)^q\, dx_1\\
		&\le C\int_{0}^L e^{-x_1(n-1)}\sigma(L)^{n-2}\left(\frac{2}{e^{x_1}+e^{-x_1}+e^{-x_1}\sigma(L)^2}\right)^{q-\frac{n}{2}+\frac{n}{2}}\, dx_1\\
		&\le Co(1)\int_{0}^L e^{-\frac{n-2}{2}x_1}\sigma(L)^{-2}\, dx_1\\
		&\le \tilde{C}o(1)\, \sigma(L)^{-2}(e^{-\frac{n-2}{2}L}+1).
	\end{aligned}
\end{equation}
Therefore, the above integral vanishes as $L\to\infty$ regardless of $\sigma(L)$.

Now, we consider the term $\sigma(L)^{-1}O(|h|_b)$.
\begin{equation}\label{eqn:estimate_lateral_surface_4}
	\begin{aligned}
		&\int_{S_L}\sigma(L)^{-1}O(|h|_b)\, d\sigma_b\\
		&\qquad\le C_1\omega_{n-2}\int_{-L}^{L}\sigma(L)^{-1}\left(\left(\cosh x_1 +\frac{e^{x_1}\sigma(L)^2}{2}\right)^2-1\right)^{-\frac{q}{2}}e^{x_1(n-2)}\sigma(L)^{n-2}\, dx_1\\
		%&\le C_2\omega_{n-2}\int_{-L}^{L}e^{x_1}\left(\cosh x_1 +\frac{e^{x_1}e^{nL}}{2}\right)^{-q}e^{x_1(n-2)}e^{\frac{n(n-2)}{2}L}\, dx_1\\
		&\qquad\le C_2\int_{-L}^L e^{x_1(n-2)}\sigma(L)^{n-3}\left(\frac{2}{e^{x_1}+e^{-x_1}+e^{x_1}\sigma(L)^2}\right)^q\, dx_1.
	\end{aligned}
\end{equation}
Since $\sigma(L)^{-1}O(|h|_b)$ can be absorbed into $VO(|h|_b)$ for $0\le x_1\le L$, we only need to estimate for $-L\le x_1\le 0$:
\begin{equation}\label{eqn:estimate_lateral_surface_5}
	\begin{aligned}
		&\int_{-L}^0 e^{x_1(n-2)}\sigma(L)^{n-3}\left(\frac{2}{e^{x_1}+e^{-x_1}+e^{x_1}\sigma(L)^2}\right)^q\, dx_1\\
		&\le C_1\int_{0}^L e^{-x_1(n-2)}\sigma(L)^{n-3}\left(\frac{2}{e^{x_1}+e^{-x_1}+e^{-x_1}\sigma(L)^2}\right)^{q-n+1+n-1}\, dx_1\\
		&\le C_2\int_{0}^L e^{-x_1(n-2)}\sigma(L)^{n-3}e^{-x_1(q-n+1)}\sigma(L)^{-n+1}\, dx_1\\
		&\le \tilde{C}\sigma(L)^{-2}(e^{-(q-1)L}+1).
	\end{aligned}
\end{equation}
For the third inequality, we used the following:
\[
	\begin{aligned}
		\left(\frac{2}{e^{x_1}+e^{-x_1}+e^{-x_1}\sigma(L)^2}\right)^{q-n+1+n-1}&\le\left(\frac{2}{e^{x_1}}\right)^{q-n+1}\left(\frac{e^{x_1}+e^{-x_1}\sigma(L)^2}{2}\right)^{-n+1}\\
		&\le 2^{q-n+1} e^{-x_1(q-n+1)}\sigma(L)^{-n+1}.
	\end{aligned}
\]
Hence, the integral from \eqref{eqn:estimate_lateral_surface_5} converges to zero as $L\to\infty$ regardless of $\sigma(L)$.\\

Note that \eqref{eqn:estimate_top_edge} and \eqref{eqn:estimate_lateral_surface_2} are the only places that require an additional condition \eqref{eqn:k_condition_1}. Combining all, we obtain
\begin{equation}
	\begin{aligned}
		p_0-p_1&=\int_{\partial C_L}\mathbb{U}(V)(\nu_0)\, d\sigma_b+o(1) &\\
		&=2\int_{F_{+,L}} V(H_b-H_g)\, d\sigma_g+o(1),\\
	\end{aligned}
\end{equation}
which proves Theorem \ref{thm:AH_mass_top_face}.
%Using the condition $q>\frac{n}{2}$, we can reduce \eqref{eqn:k_condition_1} (notice that  is the same) as
%\begin{equation}\label{eqn:k_condition_a}
%	k\ge\frac{n-2}{4}.
%\end{equation}

\begin{remark}\label{rmk:why_not_cubic}
	From the above estimates, we observe that the metric falloff rate plays an interesting role. If $q> n-1$, then the mass formula \eqref{eqn:AH_mass_top_face} holds regardless of $\sigma(L)$. If $\frac{n}{2}<q\le n-1$, then the integral on $E_{+,L}$ may not converge to zero provided $\sigma(L)$ does not increase fast enough. Similarly, if $q> n-1$, the mean curvature difference on the lateral surface does not contribute to the mass by \eqref{eqn:estimate_lateral_surface_1}-\eqref{eqn:estimate_lateral_surface_5}. As the mass formula for asymptotically flat manifolds via large coordinate cubes is considered in \cite{Miao.2019}, one may attempt to use large rectangles $\{|x_1|\le L, |x_i|\le \sigma(L)\text{ for }i=2,\ldots,n\}$ instead of cylinders $C_L$. However, our estimates suggest that the dihedral angle deficit on edges and the mean curvature differences on the lateral faces and $\{x_1=-L\}$ will not affect the mass if $q> n-1$.
\end{remark}

By setting $\sigma(L)=e^{kL}$ for some $k>0$ and using the falloff condition $q>\frac{n}{2}$, one can reduce \eqref{eqn:k_condition_1} as
\[
	k(n-2-2q)+n-1-q>0\, \Rightarrow\, k\ge\frac{n-2}{4},
\]
which is independent of $q$. Thus, we have the following corollary.

\begin{corollary}
	Let $(M^n,g),n\ge 3,$ be an asymptotically hyperbolic manifold with metric falloff rate $q>\frac{n}{2}$. Define $\Sigma_L=\{y_1=e^{-L},|\hat{y}|<e^{\frac{n-2}{4}L}\}$. Let $\nu_g$ denote the unit normal vector to $\{y_1=e^{-L}\}$ in $(M,g)$ pointing toward $\{y_1=0\}$. Let $H_g$ be the mean curvature of $\Sigma_L$ with respect to $\nu_g$ in $(M,g)$. Then, as $L\to\infty$,
	\begin{equation}
		p_0-p_1=2\int_{\Sigma_L} V(H_b-H_g)\, d\sigma_g+o(1)
	\end{equation}
	where $V=t-z_1=\frac{1}{y_1}$ on $M\setminus K$.
\end{corollary}

Now we conclude the proof of Theorem \ref{thm:AH_mass_horospheres}.
\begin{proof}[Proof of Theorem \ref{thm:AH_mass_horospheres}]
	We only need to estimate the integral on $\{x_1=L\}\setminus F_{+,L}$:
	\begin{equation}\label{eqn:horosphere estimate}
		\begin{aligned}
			&\int_{\{x_1=L\}\setminus F_{+,L}}VO(|h|_b)d\sigma_b\\
			&\qquad\le C_1\omega_{n-2}\int_{\sigma(L)}^\infty e^L\left(\left(\cosh L+\frac{e^L\rho^2}{2}\right)^2-1\right)^{-\frac{q}{2}}e^{L(n-1)}\rho^{n-2}\, d\rho\\
			%&\le C_2\omega_{n-2}\int_{e^{\frac{n}{2}L}}^\infty e^L\left(\cosh L+\frac{e^L\rho^2}{2}\right)^{-q}e^{L(n-1)}\rho^{n-2}\, d\rho\\
			&\qquad\le C_2 e^{nL}\int_{\sigma(L)}^\infty \left(\frac{2}{e^L+e^{-L}+e^{L}\rho^2}\right)^{q}\rho^{n-2}\, d\rho\\
			&\qquad\le C_3 e^{L(n-q)}\int_{\sigma(L)}^\infty \rho^{n-2-2q}\, d\rho\\
			&\qquad\le\tilde{C}e^{L(n-q)}\sigma(L)^{n-1-2q}.
		\end{aligned}
	\end{equation}
	The last inequality implies that if $q=n$, then the above integral converges to zero as $L\to\infty$ regardless of $\sigma(L)$. If $\frac{n}{2}<q<n$, we let $\sigma(L)=e^{kL}$ with $k$ satisfying
	\begin{equation}\label{eqn:k_condition_3}
		k(n-1-2q)+n-q<0.
	\end{equation}
	Considering $q>\frac{n}{2}$, we can find the inequality independent of $q$:
	\begin{equation}\label{eqn:k_condition_b}
		k\ge\frac{n}{2}.
	\end{equation}

	Let $\sigma(L)=e^{\frac{n}{2}L}$. It is clear that $\sigma(L)$ satisfies the assumptions in Theorem \ref{thm:AH_mass_top_face}. Thus, by applying Theorem \ref{thm:AH_mass_top_face}, we get
	\begin{equation}\label{eqn:last step}
		\begin{aligned}
			p_0-p_1&=2\int_{\{x_1=L\}} V(H_b-H_g)\, d\sigma_g-2\int_{\{x_1=L\}\setminus F_{+,L}}VO(|h|_b)d\sigma_b+o(1)\\
			&=2\int_{\mathcal{H}_L} V(H_b-H_g)\, d\sigma_g+o(1).
		\end{aligned}
	\end{equation}
\end{proof}

\begin{comment}
By the positive mass theorem and Proposition \ref{prop:nonlocalizable}, we can state the following optimized version of Collorary \ref{cor:rigidity_general}.

\begin{corollary}
	Let $(M^n,g),n\ge 3,$ be an asymptotically hyperbolic manifold with metric falloff rate $q>\frac{n}{2}$. Let 
	\[
		U=\left\{p\in \bigcup_{L\ge L_0}\Sigma_{L}:H_{\Sigma_L}< n-1\right\}.
	\] 
	If $\Theta(U,L)$ satisfies \eqref{eqn:theta}, then $(M,g)$ is isometric to hyperbolic space $(\mathbb{H}^n,b)$.
\end{corollary}
\end{comment}
%This proves Thereom \ref{thm:AH_mass_horospheres}, that is, the mass can be computed as the limit of the integral on large coordinate horospheres. One can also see that this formula uses the conditions $q>\frac{n}{2}$ and $k\ge\frac{n}{2}$ sharply.

\bigskip

\appendix
\section{Mass formulas via large spheres}
In this appendix, we apply \eqref{eqn:1form_formula} to coordinate spheres and obtain other mass formulas for both asymptotically flat and hyperbolic manifolds. Below, we recall the definition of asymptotically flat manifolds here (see section \ref{sec:introduction} for asymptotically hyperbolic manifolds).
\begin{definition}
	A Riemannian manifold $(M^n,g)$ is said to be asymptotically flat if there exist a compact set $K\subset M$ and a diffeomorphism $\Phi:M\setminus K\to \mathbb{R}^n\setminus B_R(0)$ such that
	\begin{enumerate}
		\item as $r\to\infty$,
		\[
			|g_{ij}-\delta_{ij}|+r|\partial_{k}g_{ij}|+r^2|\partial_l\partial_k g_{ij}|=O(r^{-q}),\quad q>\frac{n-2}{2}.
		\]
		\item $\int_M R_g d\mu_g <\infty$ where $R_g$ is the scalar curvature of $g$.
	\end{enumerate}
\end{definition}
Let $h=(\Phi^{-1})^*g-\delta$. Then the condition (1) above can be written equivalently as
\[
	|h|_\delta+r|\mathring{\nabla}h|_\delta+r^2|\mathring{\nabla}^2 h|_\delta=O(r^{-q}),\quad q>\frac{n-2}{2}.
\]
The ADM mass (or energy) \cite{ADM:1961} of $(M^n,g)$ is defined as
\begin{equation}\label{eqn:ADM mass}
	\begin{aligned}
		m_{ADM}(g)&=\frac{1}{2(n-1)\omega_{n-1}}\lim_{r\to\infty}\int_{S_r}(g_{ij,j}-g_{jj,i})\,\nu_0^i\, d\sigma_0\\
		&=\frac{1}{2(n-1)\omega_{n-1}}\lim_{r\to\infty}\int_{S_r}\mathbb{U}(1)(\nu_0)\, d\sigma_0
	\end{aligned}
\end{equation}
where the one form $\mathbb{U}$ is defined as \eqref{eqn:mass_1_form} (with the background metric $\delta$ instead of the hyperbolic metric $b$). As mentioned in the introduction, it is known that $m_{ADM}(g)$ is invariant under the choice of coordinates. (R. Bartnik \cite{Bartnik:1986dq}, P. Chru\'sciel \cite{Chrusciel:1986-geometric})

Now, we state the formulas obtained by using \eqref{eqn:1form_formula} to $S_r$. %give geometric formulas of the mass of asymptotically flat and hyperbolic manifolds respectively.
\begin{proposition}\label{prop:mass formulas}
	If $(M^n,g)$ is an asymptotically flat manifold, then the ADM mass can be computed as
	\[
		m_{ADM}(g)=\frac{1}{(n-1)\omega_{n-1}}\lim_{r\to\infty}\left[\left(\int_{S_r}\frac{n-1}{r}-H_g\right)d\sigma_g+\frac{1}{r}(|S_r|_\delta-|S_r|_g)\right]
	\]
	where $|S_r|_\delta$ and $|S_r|_g$ are the area of $S_r$ with respect to Euclidean metric and the metric $g$, respectively. Similarly, if $(M^n,g)$ is an asymptotically hyperbolic manifold, then the mass vector can be obtained from
	\[
		\begin{aligned}
			p_0&=2\lim_{r\to\infty}\left[\int_{S_r}\sqrt{1+r^2}\left(\frac{\sqrt{1+r^2}}{r}(n-1)-H_g\right)d\sigma_g+\frac{1}{r}(|S_r|_b-|S_r|_g)\right],\\
			p_i&=2\lim_{r\to\infty}\int_{S_r}z_i\left(\frac{\sqrt{1+r^2}}{r}(n-1)-H_g\right)d\sigma_g.
		\end{aligned}
	\]
\end{proposition}

\begin{remark}
	One can derive the same mass formula for asymptotically locally hyperbolic metrics. For example, if $(M^n,g)$ is an asymptotically locally hyperbolic manifold asymptotic to a model space $[r_0,\infty)\times N^{n-1}$ equipped with 
	the metric $ b = \frac{dr^2}{\kappa + r^2}+r^2 h$, where $r_0>\sqrt{|\kappa|}$ and $(N^{n-1},h)$ is an $(n-1)$-dimensional 
	space form with constant sectional curvature $\kappa$ (see \cite{Chrusciel-Herzlich:2003} for the precise definition), then the mass $p_0$ of $(M, g)$ can be computed as the following:
	\[
		p_0 = 2\lim_{r\to\infty}\left[\int_{S_r}\sqrt{\kappa +r^2}\left(\frac{\sqrt{\kappa +r^2}}{r}(n-1)-H_g\right)d\sigma_g+\frac{\kappa}{r}(|S_r|_b-|S_r|_g)\right].
	\]
\end{remark}

%To prove Proposition \ref{prop:mass formulas}, we first show the following lemma: 

\begin{proof}[Proof of Proposition \ref{prop:mass formulas}]
	First, suppose that $(M^n,g)$ is asymptotically flat and $g_0=\delta$ is Euclidean metric. On coordinate spheres $S_r$, we have
	\[
		A_{\delta}=\frac{1}{r}\delta|_{S_r}.
	\]
	By Proposition \ref{prop:1form}, we have on $S_r$,
	\[
		\begin{aligned}
			\mathbb{U}(1)(\nu_0)&=2(H_{\delta}-H_g)+|A_{\delta}|_{\delta}O(|h|_{\delta}^2)+O(|\mathring{\nabla}h|_{\delta}|h|_{\delta})-V\langle A_{\delta},h\rangle_{\delta|_{S_r}}-\mathrm{div}_{S_r} X\\
			&=2(H_{\delta}-H_g)-\frac{1}{r}\mathrm{tr}_\delta^{S_r}h-\mathrm{div}_{S_r} X+O(r^{-2q-1}).
		\end{aligned}
	\]
	By integrating it on $S_r$, we get
	\[
		\begin{aligned}
			\int_{S_r}\mathbb{U}(1)(\nu_0)\, d\sigma_{\delta}&=\int_{S_r}\left[2(H_\delta-H_g)-\frac{1}{r}\mathrm{tr}_\delta^{S_r}h+O(r^{-2q-1})\right]\, d\sigma_\delta\\
			&=\int_{S_r}2\left(H_\delta-H_g\right)d\sigma_g-\int_{S_r}\left(H_\delta-H_g+\frac{1}{r}\right)\mathrm{tr}_\delta^{S_r}h\, d\sigma_\delta + O(r^{n-2q-2})\\
			&=\int_{S_r}2\left(H_\delta-H_g\right)d\sigma_g-\frac{1}{r}\int_{S_r}\mathrm{tr}_{\delta}^{S_r}h\, d\sigma_\delta+O(r^{n-2q-2})\\
			&=\int_{S_r}2\left(H_\delta-H_g\right)d\sigma_g+\frac{2}{r}(|S_r|_{\delta}-|S_r|_{g})+O(r^{n-2q-2})
		\end{aligned}
	\]
	By taking the limit $r\to\infty$, we can conclude the formula for the ADM mass.

	Now we suppose that $(M^n,g)$ is asymptotically hyperbolic and $g_0=b$ is the hyperbolic metric. On coordinate spheres $S_r$, we have
	\[
		A_{b}=\frac{\sqrt{1+r^2}}{r}b|_{S_r}.
	\]
	We first show the formula for $p_0(g)$. Let $V=\sqrt{1+r^2}$. By direct computation, we have
	\[
		\begin{aligned}
			&dV(\nu_0)=\frac{r}{\sqrt{1+r^2}}\,dr\left(\sqrt{1+r^2}\,\frac{\partial}{\partial r}\right)=r,\\
			&(\mathrm{tr}_b^{S_r} h)dV(\nu_0)-V\langle A_b,h\rangle|_{S_r}=(\mathrm{tr}_b^{S_r} h)\left(r-\sqrt{1+r^2}\cdot\frac{\sqrt{1+r^2}}{r}\right)=-\frac{1}{r}\mathrm{tr}_b^{S_r}h.
		\end{aligned}
	\]
	By Proposition \ref{prop:1form}, we have on $S_r$,
	\[
		\begin{aligned}
			\mathbb{U}(V)(\nu_0)&=2V(H_{\delta}-H_g)-\frac{1}{r}\mathrm{tr}_\delta^{S_r}h-\mathrm{div}_{S_r} X+O(r^{-2q}).
		\end{aligned}
	\]
	By integrating it on $S_r$, we get
	\[
		\begin{aligned}
			\int_{S_r}\mathbb{U}(V)(\nu_0)\, d\sigma_{b}&=\int_{S_r}2V\left(H_b-H_g\right)d\sigma_g+\frac{2}{r}(|S_r|_{\delta}-|S_r|_{g})+O(r^{n-2q}).
		\end{aligned}
	\]
	So we conclude the formula of $p_0(g)$. Let $V=z_i$ for $i=1,\ldots,n$. Then we have
	\[
		\begin{aligned}
			&dV(\nu_0)=dz_i\left(\sqrt{1+r^2}\,\frac{\partial}{\partial r}\right)=\sqrt{1+r^2}\,\frac{z_i}{r},\\
			&(\mathrm{tr}_b^{S_r} h)dV(\nu_0)-V\langle A_b,h\rangle|_{S_r}=(\mathrm{tr}_b^{S_r} h)\left(\sqrt{1+r^2}\,\frac{z_i}{r}-z_i\cdot\frac{\sqrt{1+r^2}}{r}\right)=0.
		\end{aligned}
	\]
	By Proposition \ref{prop:1form}, we have on $S_r$,
	\[
		\begin{aligned}
			\mathbb{U}(V)(\nu_0)&=2V(H_{\delta}-H_g)-\mathrm{div}_{S_r} X+O(r^{-2q+1}).
		\end{aligned}
	\]
	By integrating it on $S_r$, we get
	\[
		\begin{aligned}
			\int_{S_r}\mathbb{U}(V)(\nu_0)\, d\sigma_{b}&=\int_{S_r}2z_i\left(H_b-H_g\right)d\sigma_g+O(r^{n-2q}).
		\end{aligned}
	\]
	This completes the proof.
\end{proof}

\bigskip

\noindent {\bf Acknowledgements.}
The authors would like to thank the organizers of the 2020 Virtual Workshop on Ricci and Scalar Curvature, which provided a lot of inspiring talks. Especially, the plenary talk of Chao Li influenced the initial idea of this work. H. C. Jang was partially supported by NSF Grant DMS-1612049 for participating in the workshop and would like to thank Christina Sormani for funding. P. Miao acknowledges the support of NSF Grant DMS-1906423. The authors are also grateful to Erwann Delay for his interest and insightful comments.

\bibliography{AH_mass_draft}
\end{document}

%% file: int_fig1_1.pdf_tex.tex
%% Creator: Inkscape 1.1.2 (b8e25be8, 2022-02-05), www.inkscape.org
%% Accompanies image file 'int_fig1_1.pdf' (pdf, eps, ps)
\begingroup%
  \makeatletter%
  \providecommand\color[2][]{%
    \errmessage{(Inkscape) Color is used for the text in Inkscape, but the package 'color.sty' is not loaded}%
    \renewcommand\color[2][]{}%
  }%
  \providecommand\transparent[1]{%
    \errmessage{(Inkscape) Transparency is used (non-zero) for the text in Inkscape, but the package 'transparent.sty' is not loaded}%
    \renewcommand\transparent[1]{}%
  }%
  \providecommand\rotatebox[2]{#2}%
  \newcommand*\fsize{\dimexpr\f@size pt\relax}%
  \newcommand*\lineheight[1]{\fontsize{\fsize}{#1\fsize}\selectfont}%
  \ifx\svgwidth\undefined%
    \setlength{\unitlength}{1484.41123662bp}%
    \ifx\svgscale\undefined%
      \relax%
    \else%
      \setlength{\unitlength}{\unitlength * \real{\svgscale}}%
    \fi%
  \else%
    \setlength{\unitlength}{\svgwidth}%
  \fi%
  \global\let\svgwidth\undefined%
  \global\let\svgscale\undefined%
  \makeatother%
  \begin{picture}(1,0.26403139)%
    \lineheight{1}%
    \setlength\tabcolsep{0pt}%
    \put(0,0){\includegraphics[width=\unitlength,page=1]{int_fig1_1.pdf}}%
    \put(0.84667104,0.2515955){\color[rgb]{0,0,0}\makebox(0,0)[lt]{\lineheight{1.25}\smash{\begin{tabular}[t]{l}$z_1=+\infty$\end{tabular}}}}%
    \put(0,0){\includegraphics[width=\unitlength,page=2]{int_fig1_1.pdf}}%
    \put(0.84711825,0.00332521){\color[rgb]{0,0,0}\makebox(0,0)[lt]{\lineheight{1.25}\smash{\begin{tabular}[t]{l}$z_1=-\infty$\end{tabular}}}}%
    \put(0,0){\includegraphics[width=\unitlength,page=3]{int_fig1_1.pdf}}%
    \put(0.97226516,0.19578918){\makebox(0,0)[lt]{\lineheight{1.25}\smash{\begin{tabular}[t]{l}$\mathcal{H}_L$\end{tabular}}}}%
    \put(0.28144155,0.18833962){\makebox(0,0)[lt]{\lineheight{1.25}\smash{\begin{tabular}[t]{l}$\mathcal{H}_L$\end{tabular}}}}%
    \put(0,0){\includegraphics[width=\unitlength,page=4]{int_fig1_1.pdf}}%
    \put(0.85206443,0.04491548){\makebox(0,0)[lt]{\lineheight{1.25}\smash{\begin{tabular}[t]{l}$\nu_g$\end{tabular}}}}%
    \put(0.22503812,0.06165456){\makebox(0,0)[lt]{\lineheight{1.25}\smash{\begin{tabular}[t]{l}$\nu_g$\end{tabular}}}}%
    \put(0,0){\includegraphics[width=\unitlength,page=5]{int_fig1_1.pdf}}%
    \put(0.65087048,0.10011918){\color[rgb]{0,0,0}\makebox(0,0)[lt]{\lineheight{1.25}\smash{\begin{tabular}[t]{l}$\mathcal{H}_L$\end{tabular}}}}%
    \put(0.5204252,0.25316365){\color[rgb]{0,0,0}\makebox(0,0)[lt]{\lineheight{1.25}\smash{\begin{tabular}[t]{l}$y_1$\end{tabular}}}}%
    \put(0.61393964,0.0434477){\color[rgb]{0,0,0}\makebox(0,0)[lt]{\lineheight{1.25}\smash{\begin{tabular}[t]{l}$(y_2,\ldots,y_n)$\end{tabular}}}}%
    \put(0,0){\includegraphics[width=\unitlength,page=6]{int_fig1_1.pdf}}%
    \put(0.47762449,0.08305104){\makebox(0,0)[lt]{\lineheight{1.25}\smash{\begin{tabular}[t]{l}$\nu_g$\end{tabular}}}}%
    \put(0,0){\includegraphics[width=\unitlength,page=7]{int_fig1_1.pdf}}%
    \put(0.48659793,0.03926419){\makebox(0,0)[lt]{\lineheight{1.25}\smash{\begin{tabular}[t]{l}$z_1=-\infty$\end{tabular}}}}%
    \put(0.52488185,0.10875512){\makebox(0,0)[lt]{\lineheight{1.25}\smash{\begin{tabular}[t]{l}$1$\end{tabular}}}}%
    \put(0,0){\includegraphics[width=\unitlength,page=8]{int_fig1_1.pdf}}%
    \put(0.15911506,0.25111662){\color[rgb]{0,0,0}\makebox(0,0)[lt]{\lineheight{1.25}\smash{\begin{tabular}[t]{l}$z_1$\end{tabular}}}}%
    \put(0.26192886,0.07246809){\color[rgb]{0,0,0}\makebox(0,0)[lt]{\lineheight{1.25}\smash{\begin{tabular}[t]{l}$(z_2,\ldots,z_n)$\end{tabular}}}}%
  \end{picture}%
\endgroup%

%% file: fig3_1.pdf_tex.tex
%% Creator: Inkscape 1.1.2 (b8e25be8, 2022-02-05), www.inkscape.org
%% Accompanies image file 'fig3_1.pdf' (pdf, eps, ps)
\begingroup%
  \makeatletter%
  \providecommand\color[2][]{%
    \errmessage{(Inkscape) Color is used for the text in Inkscape, but the package 'color.sty' is not loaded}%
    \renewcommand\color[2][]{}%
  }%
  \providecommand\transparent[1]{%
    \errmessage{(Inkscape) Transparency is used (non-zero) for the text in Inkscape, but the package 'transparent.sty' is not loaded}%
    \renewcommand\transparent[1]{}%
  }%
  \providecommand\rotatebox[2]{#2}%
  \newcommand*\fsize{\dimexpr\f@size pt\relax}%
  \newcommand*\lineheight[1]{\fontsize{\fsize}{#1\fsize}\selectfont}%
  \ifx\svgwidth\undefined%
    \setlength{\unitlength}{533.03734967bp}%
    \ifx\svgscale\undefined%
      \relax%
    \else%
      \setlength{\unitlength}{\unitlength * \real{\svgscale}}%
    \fi%
  \else%
    \setlength{\unitlength}{\svgwidth}%
  \fi%
  \global\let\svgwidth\undefined%
  \global\let\svgscale\undefined%
  \makeatother%
  \begin{picture}(1,0.6011532)%
    \lineheight{1}%
    \setlength\tabcolsep{0pt}%
    \put(0,0){\includegraphics[width=\unitlength,page=1]{fig3_1.pdf}}%
    \put(0.44981209,0.57227973){\color[rgb]{0,0,0}\makebox(0,0)[lt]{\lineheight{1.25}\smash{\begin{tabular}[t]{l}$y_1$\end{tabular}}}}%
    \put(0.83987228,0.02493537){\color[rgb]{0,0,0}\makebox(0,0)[lt]{\lineheight{1.25}\smash{\begin{tabular}[t]{l}$(y_2,\ldots,y_n)$\end{tabular}}}}%
    \put(0.4394716,0.35297449){\color[rgb]{0,0,0}\makebox(0,0)[lt]{\lineheight{1.25}\smash{\begin{tabular}[t]{l}$e^L$\end{tabular}}}}%
    \put(0.43588491,0.0851715){\color[rgb]{0,0,0}\makebox(0,0)[lt]{\lineheight{1.25}\smash{\begin{tabular}[t]{l}$e^{-L}$\end{tabular}}}}%
    \put(0.67106014,0.00718172){\color[rgb]{0,0,0}\makebox(0,0)[lt]{\lineheight{1.25}\smash{\begin{tabular}[t]{l}$\sigma(L)$\end{tabular}}}}%
    \put(0.12754418,0.01120193){\color[rgb]{0,0,0}\makebox(0,0)[lt]{\lineheight{1.25}\smash{\begin{tabular}[t]{l}$-\sigma(L)$\end{tabular}}}}%
    \put(0,0){\includegraphics[width=\unitlength,page=2]{fig3_1.pdf}}%
    \put(0.70828893,0.35754425){\color[rgb]{0,0,0}\makebox(0,0)[lt]{\lineheight{1.25}\smash{\begin{tabular}[t]{l}$C_L$\end{tabular}}}}%
    \put(0.04939649,0.16216795){\color[rgb]{0,0,0}\makebox(0,0)[lt]{\lineheight{1.25}\smash{\begin{tabular}[t]{l}$\Sigma_L$\end{tabular}}}}%
    \put(0,0){\includegraphics[width=\unitlength,page=3]{fig3_1.pdf}}%
  \end{picture}%
\endgroup%

%% file: fig5.pdf_tex.tex
%% Creator: Inkscape 1.0.1 (c497b03c, 2020-09-10), www.inkscape.org
%% Accompanies image file 'fig5.pdf' (pdf, eps, ps)
\begingroup%
  \makeatletter%
  \providecommand\color[2][]{%
    \errmessage{(Inkscape) Color is used for the text in Inkscape, but the package 'color.sty' is not loaded}%
    \renewcommand\color[2][]{}%
  }%
  \providecommand\transparent[1]{%
    \errmessage{(Inkscape) Transparency is used (non-zero) for the text in Inkscape, but the package 'transparent.sty' is not loaded}%
    \renewcommand\transparent[1]{}%
  }%
  \providecommand\rotatebox[2]{#2}%
  \newcommand*\fsize{\dimexpr\f@size pt\relax}%
  \newcommand*\lineheight[1]{\fontsize{\fsize}{#1\fsize}\selectfont}%
  \ifx\svgwidth\undefined%
    \setlength{\unitlength}{606.99782965bp}%
    \ifx\svgscale\undefined%
      \relax%
    \else%
      \setlength{\unitlength}{\unitlength * \real{\svgscale}}%
    \fi%
  \else%
    \setlength{\unitlength}{\svgwidth}%
  \fi%
  \global\let\svgwidth\undefined%
  \global\let\svgscale\undefined%
  \makeatother%
  \begin{picture}(1,0.51675022)%
    \lineheight{1}%
    \setlength\tabcolsep{0pt}%
    \put(0,0){\includegraphics[width=\unitlength,page=1]{fig5.pdf}}%
    \put(0.75017915,0.15419299){\color[rgb]{0,0,0}\makebox(0,0)[lt]{\lineheight{1.25}\smash{\begin{tabular}[t]{l}$y_1=e^{-L_0}$\end{tabular}}}}%
    \put(0.37258702,0.49139489){\color[rgb]{0,0,0}\makebox(0,0)[lt]{\lineheight{1.25}\smash{\begin{tabular}[t]{l}$y_1$\end{tabular}}}}%
    \put(0.75194427,0.00662844){\color[rgb]{0,0,0}\makebox(0,0)[lt]{\lineheight{1.25}\smash{\begin{tabular}[t]{l}$|\hat{y}|$\end{tabular}}}}%
    \put(0,0){\includegraphics[width=\unitlength,page=2]{fig5.pdf}}%
    \put(0.4356991,0.29957497){\makebox(0,0)[lt]{\lineheight{1.25}\smash{\begin{tabular}[t]{l}$|\hat{y}|=f(y_1)$\end{tabular}}}}%
    \put(0.61607553,0.088515){\makebox(0,0)[lt]{\lineheight{1.25}\smash{\begin{tabular}[t]{l}$\Sigma_L$\end{tabular}}}}%
    \put(0,0){\includegraphics[width=\unitlength,page=3]{fig5.pdf}}%
  \end{picture}%
\endgroup%

%% file: fig2.pdf_tex.tex
%% Creator: Inkscape 1.0.1 (c497b03c, 2020-09-10), www.inkscape.org
%% Accompanies image file 'fig2.pdf' (pdf, eps, ps)
\begingroup%
  \makeatletter%
  \providecommand\color[2][]{%
    \errmessage{(Inkscape) Color is used for the text in Inkscape, but the package 'color.sty' is not loaded}%
    \renewcommand\color[2][]{}%
  }%
  \providecommand\transparent[1]{%
    \errmessage{(Inkscape) Transparency is used (non-zero) for the text in Inkscape, but the package 'transparent.sty' is not loaded}%
    \renewcommand\transparent[1]{}%
  }%
  \providecommand\rotatebox[2]{#2}%
  \newcommand*\fsize{\dimexpr\f@size pt\relax}%
  \newcommand*\lineheight[1]{\fontsize{\fsize}{#1\fsize}\selectfont}%
  \ifx\svgwidth\undefined%
    \setlength{\unitlength}{557.02841931bp}%
    \ifx\svgscale\undefined%
      \relax%
    \else%
      \setlength{\unitlength}{\unitlength * \real{\svgscale}}%
    \fi%
  \else%
    \setlength{\unitlength}{\svgwidth}%
  \fi%
  \global\let\svgwidth\undefined%
  \global\let\svgscale\undefined%
  \makeatother%
  \begin{picture}(1,0.66208092)%
    \lineheight{1}%
    \setlength\tabcolsep{0pt}%
    \put(0,0){\includegraphics[width=\unitlength,page=1]{fig2.pdf}}%
    \put(0.43325424,0.62751854){\color[rgb]{0,0,0}\makebox(0,0)[lt]{\lineheight{1.25}\smash{\begin{tabular}[t]{l}$x_1$\end{tabular}}}}%
    \put(0.77713758,0.2901297){\color[rgb]{0,0,0}\makebox(0,0)[lt]{\lineheight{1.25}\smash{\begin{tabular}[t]{l}$|\hat{x}|$\end{tabular}}}}%
    \put(0.42157855,0.4484991){\color[rgb]{0,0,0}\makebox(0,0)[lt]{\lineheight{1.25}\smash{\begin{tabular}[t]{l}$L$\end{tabular}}}}%
    \put(0.41201242,0.09398807){\color[rgb]{0,0,0}\makebox(0,0)[lt]{\lineheight{1.25}\smash{\begin{tabular}[t]{l}$-L$\end{tabular}}}}%
    \put(0.63482981,0.25969004){\color[rgb]{0,0,0}\makebox(0,0)[lt]{\lineheight{1.25}\smash{\begin{tabular}[t]{l}$\sigma(L)$\end{tabular}}}}%
    \put(-0.00271741,0.25733682){\color[rgb]{0,0,0}\makebox(0,0)[lt]{\lineheight{1.25}\smash{\begin{tabular}[t]{l}$-\sigma(L)$\end{tabular}}}}%
  \end{picture}%
\endgroup%

%% file: fig3.pdf_tex.tex
%% Creator: Inkscape 1.0.1 (c497b03c, 2020-09-10), www.inkscape.org
%% Accompanies image file 'fig3.pdf' (pdf, eps, ps)
\begingroup%
  \makeatletter%
  \providecommand\color[2][]{%
    \errmessage{(Inkscape) Color is used for the text in Inkscape, but the package 'color.sty' is not loaded}%
    \renewcommand\color[2][]{}%
  }%
  \providecommand\transparent[1]{%
    \errmessage{(Inkscape) Transparency is used (non-zero) for the text in Inkscape, but the package 'transparent.sty' is not loaded}%
    \renewcommand\transparent[1]{}%
  }%
  \providecommand\rotatebox[2]{#2}%
  \newcommand*\fsize{\dimexpr\f@size pt\relax}%
  \newcommand*\lineheight[1]{\fontsize{\fsize}{#1\fsize}\selectfont}%
  \ifx\svgwidth\undefined%
    \setlength{\unitlength}{569.88774337bp}%
    \ifx\svgscale\undefined%
      \relax%
    \else%
      \setlength{\unitlength}{\unitlength * \real{\svgscale}}%
    \fi%
  \else%
    \setlength{\unitlength}{\svgwidth}%
  \fi%
  \global\let\svgwidth\undefined%
  \global\let\svgscale\undefined%
  \makeatother%
  \begin{picture}(1,0.5622811)%
    \lineheight{1}%
    \setlength\tabcolsep{0pt}%
    \put(0,0){\includegraphics[width=\unitlength,page=1]{fig3.pdf}}%
    \put(0.4207261,0.53527467){\color[rgb]{0,0,0}\makebox(0,0)[lt]{\lineheight{1.25}\smash{\begin{tabular}[t]{l}$y_1$\end{tabular}}}}%
    \put(0.785564,0.02332299){\color[rgb]{0,0,0}\makebox(0,0)[lt]{\lineheight{1.25}\smash{\begin{tabular}[t]{l}$|\hat{y}|$\end{tabular}}}}%
    \put(0.41105424,0.33015026){\color[rgb]{0,0,0}\makebox(0,0)[lt]{\lineheight{1.25}\smash{\begin{tabular}[t]{l}$e^L$\end{tabular}}}}%
    \put(0.40769948,0.0796641){\color[rgb]{0,0,0}\makebox(0,0)[lt]{\lineheight{1.25}\smash{\begin{tabular}[t]{l}$e^{-L}$\end{tabular}}}}%
    \put(0.62766768,0.00671733){\color[rgb]{0,0,0}\makebox(0,0)[lt]{\lineheight{1.25}\smash{\begin{tabular}[t]{l}$\sigma(L)$\end{tabular}}}}%
    \put(0.11929685,0.01047758){\color[rgb]{0,0,0}\makebox(0,0)[lt]{\lineheight{1.25}\smash{\begin{tabular}[t]{l}$-\sigma(L)$\end{tabular}}}}%
    \put(0,0){\includegraphics[width=\unitlength,page=2]{fig3.pdf}}%
    \put(0.81943456,0.11726541){\color[rgb]{0,0,0}\makebox(0,0)[lt]{\lineheight{1.25}\smash{\begin{tabular}[t]{l}$x_1=L$\end{tabular}}}}%
    \put(0,0){\includegraphics[width=\unitlength,page=3]{fig3.pdf}}%
    \put(0.81112883,0.36755682){\color[rgb]{0,0,0}\makebox(0,0)[lt]{\lineheight{1.25}\smash{\begin{tabular}[t]{l}$x_1=-L$\end{tabular}}}}%
    \put(0,0){\includegraphics[width=\unitlength,page=4]{fig3.pdf}}%
  \end{picture}%
\endgroup%

%% file: fig4.pdf_tex.tex
%% Creator: Inkscape 1.0.1 (c497b03c, 2020-09-10), www.inkscape.org
%% Accompanies image file 'fig4.pdf' (pdf, eps, ps)
\begingroup%
  \makeatletter%
  \providecommand\color[2][]{%
    \errmessage{(Inkscape) Color is used for the text in Inkscape, but the package 'color.sty' is not loaded}%
    \renewcommand\color[2][]{}%
  }%
  \providecommand\transparent[1]{%
    \errmessage{(Inkscape) Transparency is used (non-zero) for the text in Inkscape, but the package 'transparent.sty' is not loaded}%
    \renewcommand\transparent[1]{}%
  }%
  \providecommand\rotatebox[2]{#2}%
  \newcommand*\fsize{\dimexpr\f@size pt\relax}%
  \newcommand*\lineheight[1]{\fontsize{\fsize}{#1\fsize}\selectfont}%
  \ifx\svgwidth\undefined%
    \setlength{\unitlength}{646.61669153bp}%
    \ifx\svgscale\undefined%
      \relax%
    \else%
      \setlength{\unitlength}{\unitlength * \real{\svgscale}}%
    \fi%
  \else%
    \setlength{\unitlength}{\svgwidth}%
  \fi%
  \global\let\svgwidth\undefined%
  \global\let\svgscale\undefined%
  \makeatother%
  \begin{picture}(1,0.60612595)%
    \lineheight{1}%
    \setlength\tabcolsep{0pt}%
    \put(0,0){\includegraphics[width=\unitlength,page=1]{fig4.pdf}}%
    \put(0.40932399,0.57647806){\color[rgb]{0,0,0}\makebox(0,0)[lt]{\lineheight{1.25}\smash{\begin{tabular}[t]{l}$z_1$\end{tabular}}}}%
    \put(0.76597728,0.17100157){\color[rgb]{0,0,0}\makebox(0,0)[lt]{\lineheight{1.25}\smash{\begin{tabular}[t]{l}$|\hat{z}|$\end{tabular}}}}%
    \put(0.52250721,0.52308367){\color[rgb]{0,0,0}\makebox(0,0)[lt]{\lineheight{1.25}\smash{\begin{tabular}[t]{l}$x_1=-L$\end{tabular}}}}%
    \put(0.73668255,0.36656623){\color[rgb]{0,0,0}\makebox(0,0)[lt]{\lineheight{1.25}\smash{\begin{tabular}[t]{l}$x_1=L$\end{tabular}}}}%
    \put(0.40398961,0.07324859){\color[rgb]{0,0,0}\makebox(0,0)[lt]{\lineheight{1.25}\smash{\begin{tabular}[t]{l}$\small\sinh(-L)$\end{tabular}}}}%
    \put(0.40133027,0.2517796){\color[rgb]{0,0,0}\makebox(0,0)[lt]{\lineheight{1.25}\smash{\begin{tabular}[t]{l}$\small\sinh L$\end{tabular}}}}%
    \put(0,0){\includegraphics[width=\unitlength,page=2]{fig4.pdf}}%
  \end{picture}%
\endgroup%

%% file: fig1.pdf_tex.tex
%% Creator: Inkscape 1.0.1 (c497b03c, 2020-09-10), www.inkscape.org
%% Accompanies image file 'fig1.pdf' (pdf, eps, ps)
\begingroup%
  \makeatletter%
  \providecommand\color[2][]{%
    \errmessage{(Inkscape) Color is used for the text in Inkscape, but the package 'color.sty' is not loaded}%
    \renewcommand\color[2][]{}%
  }%
  \providecommand\transparent[1]{%
    \errmessage{(Inkscape) Transparency is used (non-zero) for the text in Inkscape, but the package 'transparent.sty' is not loaded}%
    \renewcommand\transparent[1]{}%
  }%
  \providecommand\rotatebox[2]{#2}%
  \newcommand*\fsize{\dimexpr\f@size pt\relax}%
  \newcommand*\lineheight[1]{\fontsize{\fsize}{#1\fsize}\selectfont}%
  \ifx\svgwidth\undefined%
    \setlength{\unitlength}{401.93589578bp}%
    \ifx\svgscale\undefined%
      \relax%
    \else%
      \setlength{\unitlength}{\unitlength * \real{\svgscale}}%
    \fi%
  \else%
    \setlength{\unitlength}{\svgwidth}%
  \fi%
  \global\let\svgwidth\undefined%
  \global\let\svgscale\undefined%
  \makeatother%
  \begin{picture}(1,0.77247734)%
    \lineheight{1}%
    \setlength\tabcolsep{0pt}%
    \put(0,0){\includegraphics[width=\unitlength,page=1]{fig1.pdf}}%
    \put(0.78984457,0.71346554){\color[rgb]{0,0,0}\makebox(0,0)[lt]{\lineheight{1.25}\smash{\begin{tabular}[t]{l}$x_1=-L$\end{tabular}}}}%
    \put(0,0){\includegraphics[width=\unitlength,page=2]{fig1.pdf}}%
    \put(0.0147003,0.72466238){\color[rgb]{0,0,0}\makebox(0,0)[lt]{\lineheight{1.25}\smash{\begin{tabular}[t]{l}$x_1=L$\end{tabular}}}}%
    \put(0,0){\includegraphics[width=\unitlength,page=3]{fig1.pdf}}%
  \end{picture}%
\endgroup%